\documentclass{amsart}
\PassOptionsToPackage{fleqn}{amsmath}
\RequirePackage{amsthm}
\RequirePackage{amsmath}
\RequirePackage{latexsym}
\RequirePackage{amssymb}
\RequirePackage{amscd}
\RequirePackage{epsfig}
\RequirePackage{graphics}
\RequirePackage{ifthen}
\RequirePackage{varioref}
\usepackage{changes}
\usepackage{soul}
\usepackage{xcolor}
\usepackage{color}
\usepackage[misc]{ifsym}
\usepackage{bm}
\usepackage{lipsum}
\usepackage{subfigure}

\usepackage{algorithm}
\usepackage{algorithmic}

\usepackage{mathrsfs}
\usepackage{amsthm}
\usepackage{mathrsfs}

\usepackage{epstopdf}
\definecolor{darkgreen}{rgb}{0.0625,0.64,0.0625}
\ifpdf
  \usepackage[
    pdftex,
    colorlinks,%
    linkcolor=blue,citecolor=blue,urlcolor=blue,
    hyperindex,%
    plainpages=false,%
    bookmarksopen,%
    bookmarksnumbered%
  ]{hyperref}
\else
  \usepackage{hyperref}
\fi
\usepackage{bookmark}
\allowdisplaybreaks[1]

\def\R{{\mathbb R}}

\theoremstyle{plain}
\newtheorem{thm}{Theorem}[section]

\theoremstyle{definition}
\newtheorem{rem}[thm]{Remark}

\newtheorem{defn}[thm]{Definition}

\def\Bbb#1{{\mathbb#1}}

\def\R{\Bbb R}

\numberwithin{equation}{section}
\newcommand\blfootnote[1]{%
 \begingroup
 \renewcommand\thefootnote{}\footnote{#1}%
 \addtocounter{footnote}{-1}%
 \endgroup
 }

\begin{document}

\title[Forming periodic orbits with similarity signature curves]{Similarity signature curves for forming periodic orbits in the Lorenz system}
   \dedicatory{Dedicated to  Professor Peter J. Olver on the occasion of
    his 70th birthday}     % REMOVE THIS LINE IF NOT NECESSARY

    \author[J. D. Li]{Jindi Li}% ${}^*$}%\footnote{Corresponding author: yangyun@mail.neu.edu.cn}
    \address{ Jindi Li\newline\indent
     Department of Mathematics, Northeastern University, Shenyang, 110819, P.R. China}
    \email{lijindi0717@hotmail.com}

    \author[Y. Yang]{Yun Yang${}^*$}\blfootnote{${}^*$~Corresponding author: yangyun@mail.neu.edu.cn}
    \address{ Yun Yang\newline\indent
     Department of Mathematics, Northeastern University, Shenyang, 110819, P.R. China}
    \email{yangyun@mail.neu.edu.cn}

\begin{abstract}
In this paper, we systematically investigate the short periodic orbits of the Lorenz system by the aid of the similarity signature curve, and a novel method to find the short-period orbits of the Lorenz system is proposed. The similarity invariants are derived by the equivariant moving frame theory and then the similarity signature curve occurs along with them. The similarity signature curve of the Lorenz system presents a more regular behavior than the original one. By combining the sliding window method, the quasi-periodic orbits can be detected numerically, all periodic orbits with period $p \leqslant 8$ in the Lorenz system are found, and their period lengths and symbol sequences are calculated.
\end{abstract}

%\subjclass[2010]{53A15, 53A55, 53E40, 35K52.}

%Please type here List of Keywords for your article separated by semicolon.

%\Classification{53A15;\; 53A55} % e.g. 35A30; 81Q05
% For 2010 Mathematics Subject Classification see http://www.ams.org/mathscinet/msc/msc2010.html

%\linenumbers
\maketitle
\section{Introduction}
The Lorenz system was initially derived from the Oberbeck-Boussinesq approximation by Lorenz in 1963 \cite{Lorenz}. Since then, chaos, as an interesting phenomenon in nonlinear dynamical systems has become a widely-studied topic of great interest to specialists and non-specialists alike. Tucker strictly proved the existence of the Lorenz attractor \cite{exi}, which laid the foundation for subsequent research. In fact, there are infinite periodic orbits in chaotic attractors, and unstable periodic orbits can be used to characterize the chaotic state of the system. Therefore, exploring the existence and position of periodic orbits in chaotic systems attracts the researchers' attention.

In chaos, finding all possible symmetries, i.e., self-equivalences or self-congruences
provides an effective way to understand the mathematical properties of the whole system.
Generally,  the efficient solution of an equivalence problem rests on investigating some related invariants.
In fact, invariants supply us with a moduli space of a certain kind of geometric object
under group transformations. To obtain the invariants of submanifolds for general transformation groups, Fels and Olver
formulated a new, powerful, constructive approach to the equivariant moving frame theory \cite{mft}.
Let $M$ be a smooth, meaning $C^{\infty}$ manifold with a sufficiently high order jet bundle, and suppose $G$ is a Lie group acting smoothly on $M$ via prolongation.
The practical construction of a moving frame is based on a choice of cross section $K\subset M$ to the group orbits \cite{mft, pro}.
For example, the right-equivariant moving frame map $\rho:M\rightarrow G$ is locally uniquely obtained by
solving the normalization equations $g\cdot z\in K$ for the group parameters $g=\rho(z)$ in terms of the point $z\in M$.
Then substituting the moving frame formulae $g=\rho(z)$ into the unnormalized components leads to the fundamental invariants $I_1(z),\cdots, I_{m-r}(z)$,
where $r$ and $m$ denote the dimensions of $G$ and $M$, respectively.
Then every invariant $I(z)$ can be locally uniquely expressed as a function of the fundamental invariants \cite{fram}.

According to Cartan's main idea,   two regular submanifolds
are (locally) equivalent if and only if they have
identical syzygies among all their differential
invariants. More generally, as a consequence of the Fundamental Basis Theorem \cite{equi}, one can verify
that, for any Lie group action, the entire algebra of differential
invariants can be generated from a finite number of low order invariants by repeated
invariant differentiation. These typically include the generating differential invariants
$I^1,I^2,\cdots I^l$ as well as a certain finite collection of their invariant derivatives $I^{\nu}_{,J}$.
These differential invariants serve to define a signature map $\sigma:S\to\Sigma\subset\R^{N}$
whose image is a
differential invariant signature of the original submanifold $S$. Under certain regularity assumptions,
the signature solves the equivalence problem: two $p$-dimensional submanifolds
are locally equivalent under the transformation group if and only if they have identical
signatures \cite{pro}.

The signature curve $\Sigma\subset\R^2$ of the Euclidean plane curve $C\subset\R^2$ is parametrized by the two lowest order differential
invariants \cite{exten}
\begin{equation*}
	\chi:C\rightarrow \Sigma=\left\{\left(\kappa,\displaystyle{\frac{d\kappa}{ds}}\right)\right\}\subset\R^2.
\end{equation*}

Similarly, the differential invariant signature for the Euclidean space curve $C\subset\R^3$ is defined by
\begin{equation*}
	\Sigma=\left\{(\kappa,\kappa_s,\tau)\right\}\subset\R^3,
\end{equation*}
where $\kappa$ and $\tau$ are the curvature and torsion of the curve $C$, respectively.

The above-mentioned two types of the signature curve are invariant to both the translation and rotation of the curve.
This concept has received considerable attentions in computer vision, mostly in the issues of object recognition and symmetry detection. Calabi et al. introduced a new paradigm, the differentially invariant signature curve or manifold, for the invariant recognition of visual objects, and discussed various aspects of the numerical computation of signatures and their applications \cite{dniscaor}. Boutin gave the signature formulas which are invariant under the action of the Euclidean group and affine group respectively. Another important property of the signature curve is that it is valid for any fine partition of a given curve \cite{Num}. Hoff and Olver proposed an automated solution to two-dimensional picture puzzles using the signature curve \cite{puzzle}. Bruckstein and Snaked presented a general framework for skew-symmetric detection, based on invariant planar curve descriptions and successfully applied it to mirror-symmetric polygons distorted by affine and projective viewing transformations \cite{Ske}. The similarity signature curve is a generalization of the signature curve under similarity transformation, which exhibits invariance to translation, rotation, and scaling. Bruckstein and Netravali generalized the curvature versus arc length representation that is invariant under Euclidean motions to general geometric viewing transformations, such as similarity and affine transformations. Such results can simplify the problem of model-based planar object recognition under partial occlusion to the matching of locally invariant signature functions \cite{Ond}. Ghorbel et al. completed image reconstruction using the similarity signature curve \cite{img}. As a matter of fact, the Lorenz attractor exhibits fractal structure and self-similarity, and simpler structures can be modeled using the similarity signature curve.

Periodic orbit theory is a powerful tool for analyzing the dynamic behavior of chaotic systems, which can effectively calculate the average value of physical quantities of dynamic systems \cite{app, cyc}. Many methods for finding periodic orbits have  been proposed. The interval method provides a simple computational test for the uniqueness, existence, and nonexistence of the zeros of a map within a given interval vector \cite{Kra} , and supplies research support for the possibility of complete analysis of piecewise continuous systems by combination with the Poincar\'e map \cite{fen}. Galias and Zgliczy\'nski introduced the Krawczyk operator to prove the existence of periodic orbits of infinite-dimensional discrete dynamical systems and found all periodic orbits of a given period \cite{Gal}.  In addition, the traditional approach is to find periodic orbits by a general technique based on the construction of a graph describing the dynamics of the system. Galias and Tucker successfully studied the existence of periodic orbits for a particular type of system by using symbolic dynamics combined with the interval method \cite{qujian}. Moreover, a variational principle was utilized to determine the periodic orbits of a class of continuous systems and unstable spatiotemporally periodic solutions of extended systems, such as the R\"ossler attractor \cite{flo}.

It has also been suggested that it is difficult to obtain an efficient set of Poincar\'e sections when finding the periodic orbits of turbulent or high-dimensional flows. In response to this situation, some people proposed the variational method. Lan replaced the Poincar\'e section with the variational method and verified the existence of periodic orbits in the numerical calculation \cite{Phy}.

In this paper, we present a method that can be employed to model periodic orbits of the Lorenz system, which relies on the similarity signature curve of the system, the interval method, and the sliding window method. The similarity signature curve of the system is calculated, and then the trajectory of the Lorenz system is segmented by the sliding window method. Each quasi-periodic orbit is verified by using the Krawczyk operator. This makes it possible to find periodic orbits in the trajectories between the segmentation points.

The structure of this paper is as follows. In Sec.2, we introduce the Lorenz equation which sketches the chaotic behavior, together with the typical parameters. On the other hand, the similarity invariants are derived by using the equivariant moving frame theory, and the similarity signature curve of the Lorenz attractor is described. In Sec. 3, the sliding window method is improved, and the algorithm based on the similarity signature curve to segment trajectories is proposed, which can obtain the sequence of segmentation points. The trajectories between these points form quasi-periodic orbits. In Sec. 4, all periodic orbits with period $p \leqslant 8$ in the Lorenz system are shown, and the period number, period length, and motion trajectory of each periodic orbit are listed. The results of this paper are summarized in Sec. 5.

\subsection*{Acknowledgements}  This work was supported by the Fundamental Research Funds for the Central Universities under grant-N2104007, and the second author would also like to express his deep gratitude to Professor Peter J. Olver for his encouragement and help during his stay in School of Mathematics, University of Minnesota as a Visiting Professor, while part of this work was completed.

\section{Lorenz system and its similarity signature curve}\label{sec-back}
\subsection{Lorenz system dynamics}\label{mvf}\

The Lorenz system, as the benchmark system for chaotic dynamics, is the first chaotic dissipative system discovered in numerical experiments. In this section, we discuss the relationship between the similarity signature curve and the periodic orbit. In details, the Lorenz equations can be described as a system of three differential equations, namely,
\begin{equation*}
	\left\{ \begin{array}{l}
		\dot x\left( t \right) = \sigma\big( {y\left( t \right) - x\left( t \right)} \big),\\
		\dot y\left( t \right) = rx\left( t \right) - y\left( t \right) - x\left( t \right)z\left( t \right),\\
		\dot z\left( t \right) = x\left( t \right)y\left( t \right) - \eta z\left( t \right),
	\end{array} \right.
\end{equation*}
where $\dot x$ denotes the derivative of $x$ with respect to time $t$. The positive parameters $\sigma$, $\eta$, and $r$ raise up from the physical context and Lorenz chose in his original work
\begin{center}
	$\sigma = 10$,$\quad$ $\eta =\displaystyle{\frac{{\text{8}}}{{\text{3}}}}$, $\quad$$r = 28$.
\end{center}
In this paper, we consider the numerical simulation of the Lorenz system with above parameter values. Fig. \ref{Fig:fige1} shows a numerical approximation with the help of the computer software Matlab.

\begin{figure*}[htpb]
	\centering
	\includegraphics[height=0.47\textwidth,width=0.6\textwidth]{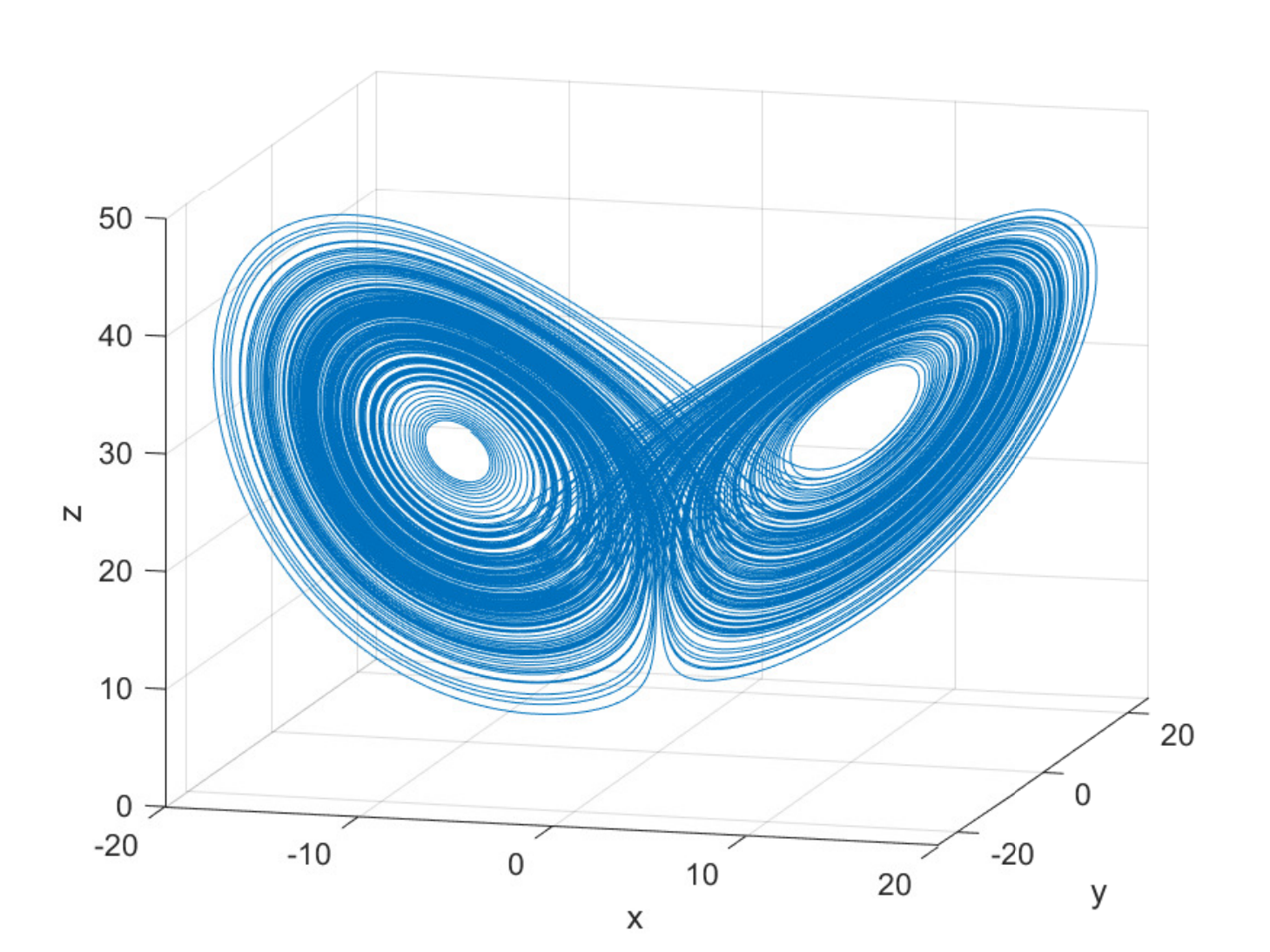}
	\caption{A trajectory of the Lorenz system.}
	\label{Fig:fige1}
\end{figure*}

First of all, the Lorenz equations contain symmetry. Applying the coordinate transformation
\begin{equation*}
	\left( {x,y,z} \right) \to \left( { - x, - y,z} \right),
\end{equation*}
the equations remain the same.

For $r < 1$ (and $\sigma$, $\eta$ arbitrarily), the point $\left( { 0, 0, 0} \right)$ is the only fixed point, which is automatically asymptotic stable. It holds that all solutions converge to the origin as $t \to \infty $.
At $r = 1$, it occurs a supercritical Pitchfork–bifurcation and two new fixed points appear. They fulfill
\begin{equation*}
	x = y = \sqrt {b\left( {r - 1} \right)}, \qquad z = r - 1.
\end{equation*}
When $r > {r_H} \approx 24.737$, the Lorenz system, which exhibits chaotic behavior, is the object of our study.

In addition, its stability matrix can be obtained from the Lorenz equations
\begin{equation*}
	M = \left[ {\begin{array}{*{20}{c}}
			{ - \sigma}&{\sigma}&0\\
			{r - z_0}&{ - 1}&{ - x_0}\\
			y_0&x_0&{ - \eta}
	\end{array}} \right],
\end{equation*}
where $x_0$, $y_0$, and $z_0$ are the coordinates of the stable points. The eigenvalues at stable points indirectly describe the degree to which trajectories in its vicinity are attracted or repelled. The eigenvalue at the origin is $\left( {{\lambda ^{\left( 1 \right)}},{\lambda ^{\left( 2 \right)}},{\lambda ^{\left( 3 \right)}}} \right) = \left( {11.83, - 2.67, - 22.83} \right)$, and the eigenvalues at the other two symmetrical stable points are $\left( {{\mu ^{\left( 1 \right)}} \pm {\omega ^{\left( 1 \right)}}i,{\lambda ^{\left( 3 \right)}}} \right) = \left( {0.094 \pm 10.19i, - 13.85} \right)$. The unstable eigenplane of the origin is spanned by Re ${e^{\left( 1 \right)}}$ and Im ${e^{\left( 1 \right)}}$. The Lorenz flow has an unstable eigenvector ${e^{\left( 1 \right)}}$ and two stable eigenvectors ${e^{\left( 2 \right)}}$, ${e^{\left( 3 \right)}}$ near the origin.

In \cite{book}, the motion of the Lorenz attractor is given in numerical terms. The periodic time scale in the neighborhood of the equilibrium $\left(x, y, z \right) = \left(0, 0, 0 \right)$ is of order $ \approx T$. The contraction/expansion radially by the multiplier ${\Lambda _{radial}}$, and by the multiplier ${\Lambda _{j}}$ along the ${e^{\left( j \right)}}$ eigen-direction per a turn of the spiral:
\begin{equation*}
	T = 2\pi /\omega , \quad {\Lambda _{radial}} = {e^{T\mu }}, \quad{\Lambda _j} = {e^{T\mu \left( j \right)}}.
\end{equation*}

The trajectory period around the critical point is about $T = 2\pi /10.19 \approx 0.62s$. Near the two stable points the unstable manifold trajectories spiral out, with very small radial per-turn expansion multiplier ${\Lambda ^1} = \exp \left( {0.094T} \right) \approx  1.06$, and very strong contraction multiplier ${\Lambda ^3} = \exp \left( { - 13.85T} \right) \approx 1.957 \times {10^{ - 4}}$ onto the unstable manifold. This contraction confines the Lorenz attractor to a two-dimensional surface. In the neighborhood of the origin, the trajectories have an extremely strong contraction along the ${e^{\left( 3 \right)}}$ direction and the slowest contraction along the ${e^{\left( 2 \right)}}$ direction. In the plane, the expansion of the attractor in the ${e^{\left( 1 \right)}}$ direction takes precedence over the contraction in the ${e^{\left( 2 \right)}}$ direction, which makes it difficult for a few trajectories can approach the origin.

\subsection{Similarity signature curve}\

The goal of this part is to establish a basic functional relation or syzygy among the
differential invariants under the action of the similarity group $\mathrm{Sim}(3)=\{
\lambda Ax + b\}$,
where $A$ is a real orthogonal $3\times3$ matrix, $b\in\R^3$ is a real vector and $\lambda>0$ is a real number.
In fact, the similarity invariants have been obtained in \cite{chaos, sfc}. Here we mainly use the equivariant moving frame theory to build the syzygy and the similarity signature curve.

We denote points in $\R^3$ by $z=\left(x,u,v\right)\in\R^3$, with $z(t)=\left(x(t),u(t),v(t)\right)$
for $t\in I\subset\R$ being a smoothly parametrized curve.

Let us consider the action
\begin{equation*}
	\left(X,U,V\right)=Z=g\cdot z=g\cdot\left(x,u,v\right)
\end{equation*}
of the similarity group $\mathrm{Sim}(3)$ in $\R^3$, so that
\begin{equation}\label{sim-tranfm}
	X=k\alpha\cdot z+a,\quad U=k\beta\cdot z+b,\quad V=k\gamma\cdot z+c,
\end{equation}
where $k>0, a,b,c$ are real numbers, $\alpha,\beta, \gamma$ are the rows of the real orthogonal $3\times3$  matrix.

To construct a (right) equivariant moving frame, we prolong the $\mathrm{Sim}(3)$ action to
the curve jet spaces $\mathrm{J}^n$, which has local coordinates
\begin{equation*}
	z^{(n)}=\left(x,u,v,u_x,v_x,u_{xx},v_{xx},\cdots,u_n,v_n\right).
\end{equation*}
The formulae for the prolonged action on $\mathrm{J}^n$ are provided by implicit differentiation, based on
\begin{equation*}
	dX=k\alpha\cdot z_tdt,\qquad D_X=\displaystyle{\frac{1}{k\alpha\cdot z_t}\frac{d}{dt}}.
\end{equation*}
Then we have
\begin{gather*}
	U_X=\frac{\beta\cdot z_t}{\alpha\cdot z_t},\qquad V_X=\frac{\gamma\cdot z_t}{\alpha\cdot z_t},\\
	U_{XX}=\frac{\left(\alpha\cdot z_t\right)\left(\beta\cdot z_{tt}\right)-\left(\alpha\cdot z_{tt}\right)\left(\beta\cdot z_{t}\right)}{k\left(\alpha\cdot z_t\right)^3},\\
	V_{XX}=\frac{\left(\alpha\cdot z_t\right)\left(\gamma\cdot z_{tt}\right)-\left(\alpha\cdot z_{tt}\right)\left(\gamma\cdot z_{t}\right)}{k\left(\alpha\cdot z_t\right)^3}.
\end{gather*}

To compute the equivariant moving frame, we must normalize the $\mathrm{dim\{Sim(3)\}} = 7$
independent group parameters by setting 7 of the transformed jet variables equal to conveniently
chosen constants. This corresponds to the choice of a cross section to the
prolonged group orbits, or, equivalently, to placing the curve in normal form, \cite{equi}.
The standard cross section that produces the classical moving frame is given by
\begin{center}
	$x=u=v=u_x=v_x=v_{xx}=0, \qquad u_{xx}=1.$
\end{center}
The determination of the moving frame associated with the above cross-section  relies on
solving the corresponding normalization equations
\begin{equation*}
	X=U=V=U_X=V_X=V_{XX}=0, \qquad U_{XX}=1,
\end{equation*}
for the group parameters $g\in\mathrm{Sim}(3)$, using the explicit formulas for the prolonged
Euclidean transformations that were obtained by implicit differentiation.

The order zero normalizations
$X = U = V = 0$ prescribe the translation parameters of the group element
$g\in\mathrm{Sim}(3)$. Since they play no further role in the prolonged action or other moving frame
formulae, we can effectively ignore them from here on.
After some obvious simplification, we find
\begin{equation*}
	\beta\cdot z_t=0,\qquad \gamma\cdot z_t=0, \qquad \gamma\cdot z_{tt}=0,\qquad \displaystyle{\frac{\beta\cdot z_{tt}}{k\left(\alpha\cdot z_t\right)^2}=1}.
\end{equation*}
We impose the nondegeneracy condition $|z_t\wedge z_{tt}|\neq0$.
Therefore, the right-equivariant moving frame induced by the cross-section  is given
by
\begin{align}
	\begin{split}\label{mfv}
		&\gamma=\frac{z_t\wedge z_{tt}}{\left|z_t\wedge z_{tt}\right|},\quad \beta=\frac{|z_t|^2z_{tt}-(z_t\cdot z_{tt})z_t}{\left||z_t|^2z_{tt}-(z_t\cdot z_{tt})z_t\right|}=\frac{|z_t|^2z_{tt}-(z_t\cdot z_{tt})z_t}{|z_t||z_t\wedge z_{tt}|},\\
		&\alpha=\beta\wedge\gamma=\frac{z_t}{|z_t|},\quad k=\frac{\beta\cdot z_{tt}}{\left(\alpha\cdot z_t\right)^2}=\frac{|z_t\wedge z_{tt}|}{|z_t|^3}.
	\end{split}
\end{align}
As always, a complete system of functionally independent differential invariants is
obtained by invariantization, that is, substituting the moving frame formulae  into the
unnormalized transformation rules. All other differential invariants are obtained by differentiating the curvature
and torsion with respect to the similarity length element,
\begin{equation*}
	d\tilde{s}=\displaystyle{\frac{|z_t\wedge z_{tt}|}{|z_t|^2}dt},
\end{equation*}
which comes from substituting the moving frame formula for $\alpha$ and $k$ into $dX$.
\begin{align*}
	U_{XXX}=&\frac{1}{k^2\left(\alpha\cdot z_t\right)^5}\times\Big((\beta\cdot z_{ttt})(\alpha\cdot z_t)^2\\
	&\quad-(\beta\cdot z_t)(\alpha\cdot z_t)(\alpha\cdot z_{ttt})-3(\alpha\cdot z_{tt})(\beta\cdot z_{tt})(\alpha\cdot z_t)+3(\alpha\cdot z_{tt})^2(\beta\cdot z_t)\Big),\\
	V_{XXX}=&\frac{1}{k^2\left(\alpha\cdot z_t\right)^5}\times\Big((\gamma\cdot z_{ttt})(\alpha\cdot z_t)^2\\
	&\quad-(\gamma\cdot z_t)(\alpha\cdot z_t)(\alpha\cdot z_{ttt})-3(\alpha\cdot z_{tt})(\gamma\cdot z_{tt})(\alpha\cdot z_t)+3(\alpha\cdot z_{tt})^2(\gamma\cdot z_t)\Big).\\
\end{align*}
The generating differential invariants are the similarity curvature $\tilde{\kappa}$,
obtained from $u_{xxx}$, and the similarity torsion $\tilde{\tau}$ , obtained
from $v_{xxx}$.
By using \eqref{mfv}, a short computation
produces the required expressions for the curvature and torsion invariants
\begin{align*}
	\tilde{\kappa}&=\iota(u_{xxx})=\frac{(z_t\wedge z_{ttt})\cdot(z_t\wedge z_{tt})|z_t|^2-3(z_t\cdot z_{tt})|z_t\wedge z_{tt}|^2}{|z_t\wedge z_{tt}|^3},\\
	\tilde{\tau}&=\iota(v_{xxx})=\frac{|z_t|^3[z_t,z_{tt},z_{ttt}]}{|z_t\wedge z_{tt}|^3}.
\end{align*}

\begin{rem} In fact, taking $k=1$ in \eqref{sim-tranfm} and using cross section $x=u=v=u_x=v_x=v_{xx}=0$, a similar process yields the Euclidean arc element
	$ds=|z_t|dt$, the Euclidean curvature $\kappa=\iota(u_{xx})=\displaystyle{\frac{|z_t\wedge z_{tt}|}{|z_t|^3}}$ and the Euclidean torsion $\tau=\iota(v_{xxx})/\kappa=\displaystyle{\frac{[z_t,z_{tt},z_{ttt}]}{|z_t\wedge z_{tt}|^2}}$. At the same time, $\iota(u_{xxx})=\kappa_s$.
\end{rem}
By using the inductive construction method of moving frames introduced by Kogan \cite{ko}, we can see $\mathrm{Sim}(3)= B\cdot \mathrm{SE}(3)$, where
$B=\left(
\begin{array}{ccc}
	k & 0 & 0 \\
	0 & k & 0 \\
	0 & 0 & k \\
\end{array}
\right)
$
and
$B\cap \mathrm{SE}(3)=\{I,-I\}$ is finite. Now we prolong the action of $B$ up to third order
\begin{align*}
	x\to kx, \quad u\to ku,\quad v\to kv, \quad u_x\to u_x,\quad v_x\to v_x,\\
	u_{xx}\to \displaystyle{\frac{u_{xx}}{k}},\quad v_{xx}\to \displaystyle{\frac{v_{xx}}{k}},\quad u_{xxx}\to \displaystyle{\frac{u_{xxx}}{k^2}},\quad v_{xxx}\to \displaystyle{\frac{v_{xxx}}{k^2}}.
\end{align*}
Restricting these transformations to the Euclidean cross section $K_E=\{x=u=v=u_x=v_x=v_{xx}=0\}$, we obtain
\begin{gather*}
	u_{xx}\to \frac{u_{xx}}{k}, \quad u_{xxx}\to \frac{u_{xxx}}{k^2},\quad v_{xxx}\to \frac{v_{xxx}}{k^2}.
\end{gather*}
The above action is free on the open subset $\{z\in K_E|u_{xx}\neq0\}$, and we choose the cross section $K=\{z\in K_E|u_{xx}=1\}$
to the orbits of $B$ on $K_E$. This produces a moving frame $\rho_B:K_E\to B$, that is,
\begin{equation*}
	k=u_{xx}.
\end{equation*}
Thus we obtain the expression of the similarity invariants $\tilde{\kappa}$ and $\tilde{\tau}$ in terms of Euclidean invariants
\begin{center}
	$\tilde{\kappa}=\displaystyle{\frac{\kappa_s}{\kappa^2}},\qquad \tilde{\tau}=\displaystyle{\frac{\tau}{\kappa}}$.
\end{center}

Let us introduce the following basis for the infinitesimal
generators in the Lie algebra of $\mathrm{Sim}(3)$
\begin{gather*}
	\mathbf{v}_1=\partial_x,\quad \mathbf{v}_2=\partial_u,\quad \mathbf{v}_3=\partial_v,\quad \mathbf{v}_4=v\partial_u-u\partial_v,\\
	\mathbf{v}_5=-u\partial_x+x\partial_u,\quad \mathbf{v}_6=-v\partial_x+x\partial_v,\quad \mathbf{v}_7=x\partial_x+u\partial_u+v\partial_v.
\end{gather*}
The prolonged infinitesimal generators of the group $\mathrm{Sim}(3)$ action on curve jets are
\begin{align*}
	\centering
	\mathrm{pr}\mathbf{v}_1&=\;\partial_x,\qquad \mathrm{pr}\mathbf{v}_2\;=\;\partial_u, \qquad \mathrm{pr}\mathbf{v}_3\;=\;\partial_v,\\
	\mathrm{pr}\mathbf{v}_4&=\;v\partial_u-u\partial_v+v_x\partial_{u_x}-u_x\partial_{v_x}+v_{xx}\partial_{u_{xx}}-u_{xx}\partial_{v_{xx}}+\cdots,\\
	\mathrm{pr}\mathbf{v}_5&=\;-u\partial_x+x\partial_u+(1+u^2_x)\partial_{u_x}+u_xv_x\partial_{v_x}+3u_xu_{xx}\partial_{u_{xx}}\\
	&\qquad+\left(u_{xx}v_{x}+2u_{x}v_{xx}\right)\partial_{v_{xx}}+\left(3u^2_{xx}+4u_xu_{xxx}\right)\partial_{u_{xxx}}\\
	&\qquad+\left(u_{xxx}v_{x}+3u_{xx}v_{xx}+3u_xv_{xxx}\right)\partial_{v_{xxx}}+\cdots,\\
	\mathrm{pr}\mathbf{v}_6&=\;-v\partial_x+x\partial_v+v_xu_x\partial_{u_x}+\left(1+v^2_x\right)\partial_{v_x}+\left(v_{xx}u_x+2v_xu_{xx}\right)\partial_{u_{xx}}\\
	&\qquad+3v_xv_{xx}\partial_{v_{xx}}+\left(v_{xxx}u_x+3v_{xx}u_{xx}+3v_xu_{xxx}\right)\partial_{u_{xxx}}\\
	&\qquad+\left(3v^2_{xx}+4v_xv_{xxx}\right)\partial_{v_{xxx}}+\cdots,\\
	\mathrm{pr}\mathbf{v}_7&=\;x\partial_x+u\partial_u+v\partial_v-u_{xx}\partial_{u_{xx}}-v_{xx}\partial_{v_{xx}}-2u_{xxx}\partial_{u_{xxx}}-2v_{xxx}\partial_{v_{xxx}}+\cdots.
\end{align*}
The invariant arc length derivative $\mathcal{D}=\iota(D_x)$ of any differential invariant $I = \iota(F)$ obtained
by invariantizing a differential function $F$ is specified by the recurrence relation
\begin{equation}\label{arcD}
	\mathcal{D}I=\mathcal{D}\iota(F)=\iota(D_xF)+\sum_{i=1}^7R^i\iota(\mathrm{pr}\mathbf{v}_i(F)),
\end{equation}
where $R^1,R^2,\cdots, R^7$ are the Maurer-Cartan invariants. To determine their formulas,
we write out \eqref{arcD} for the seven phantom invariants which come from the cross section variables $x$, $u$, $v$, $u_x$, $v_x$, $u_{xx}$, $v_{xx}$
\begin{gather*}
	0=\mathcal{D}\iota(x)=\iota(1)+R^1,\quad
	0=\mathcal{D}\iota(u)=R^2,\quad
	0=\mathcal{D}\iota(v)=R^3,\quad
	0=\mathcal{D}\iota(u_x)=1+R^5,\\
	0=\mathcal{D}\iota(v_x)=R^6,\quad
	0=\mathcal{D}\iota(u_{xx})=\tilde{\kappa}-R^7,\quad
	0=\mathcal{D}\iota(v_{xx})=\tilde{\tau}-R^4.
\end{gather*}
Thus, the general recurrence relation \eqref{arcD} becomes
\begin{equation*}
	\mathcal{D}\iota(F)=\iota(D_xF)-\iota(\mathrm{pr}\mathbf{v}_1(F))+\tilde{\tau}\iota(\mathrm{pr}\mathbf{v}_4(F))-\iota(\mathrm{pr}\mathbf{v}_5(F))+\tilde{\kappa}\iota(\mathrm{pr}\mathbf{v}_7(F)).
\end{equation*}
Then it is obvious that
\begin{equation*}
	\iota(u_{xxxx})=\tilde{\kappa}_{\tilde{s}}+3+2\tilde{\kappa}^2-\tilde{\tau}^2,\quad \iota(v_{xxxx})=\tilde{\tau}_{\tilde{s}}+3\tilde{\kappa}\tilde{\tau}.
\end{equation*}
\begin{defn}
	The differential invariant signature for the similarity space curve $C\subset\R^3$ is is parametrized by the three
	lowest order differential invariants
	\begin{center}
		$\Sigma=\left\{(\tilde{\kappa},\tilde{\kappa}_{\tilde{s}},\tilde{\tau})\right\}\subset\R^3$,
	\end{center}
	where $\tilde{s}$ is the similarity arc length parameter, $\tilde{\kappa}$ and $\tilde{\tau}$ are the similarity curvature and torsion of the curve $C$, respectively.
\end{defn}

In the actual calculation, the curve is composed of discrete points distributed in space. In \cite{dniscaor}, the authors provided an approach to the numerical approximation of differential invariants, based on a suitable combination of invariants of the group action to compute differential invariant signatures numerically in a fully group-invariant manner. In Fig. \ref{Fig:fige2}, suppose ${P_{i - 1}}$, ${P_{i}}$, ${P_{i + 1}}$, and ${P_{i + 2}}$ are four successive points on a three-dimensional curve  $C$. The Euclidean distance between points ${P_{i}}$ and ${P_{j}}$ is $\mathrm{dist}\left( {{P_i},{P_j}} \right)$. Let
\begin{gather*}
	a = \mathrm{dist}\left( {{P_{i - 1}},{P_i}} \right), \quad b = \mathrm{dist}\left( {{P_{i}},{P_{i + 1}}} \right), \quad c = \mathrm{dist}\left( {{P_{i - 1}},{P_{i + 1}}} \right),\\
	d = \mathrm{dist}\left( {{P_{i + 1}},{P_{i + 2}}} \right), \quad e = \mathrm{dist}\left( {{P_{i}},{P_{i + 2}}} \right), \quad f = \mathrm{dist}\left( {{P_{i - 1}},{P_{i + 2}}} \right).
\end{gather*}
\begin{figure}[htpb]
	\centering
	\includegraphics[height=0.47\textwidth,width=0.6\textwidth]{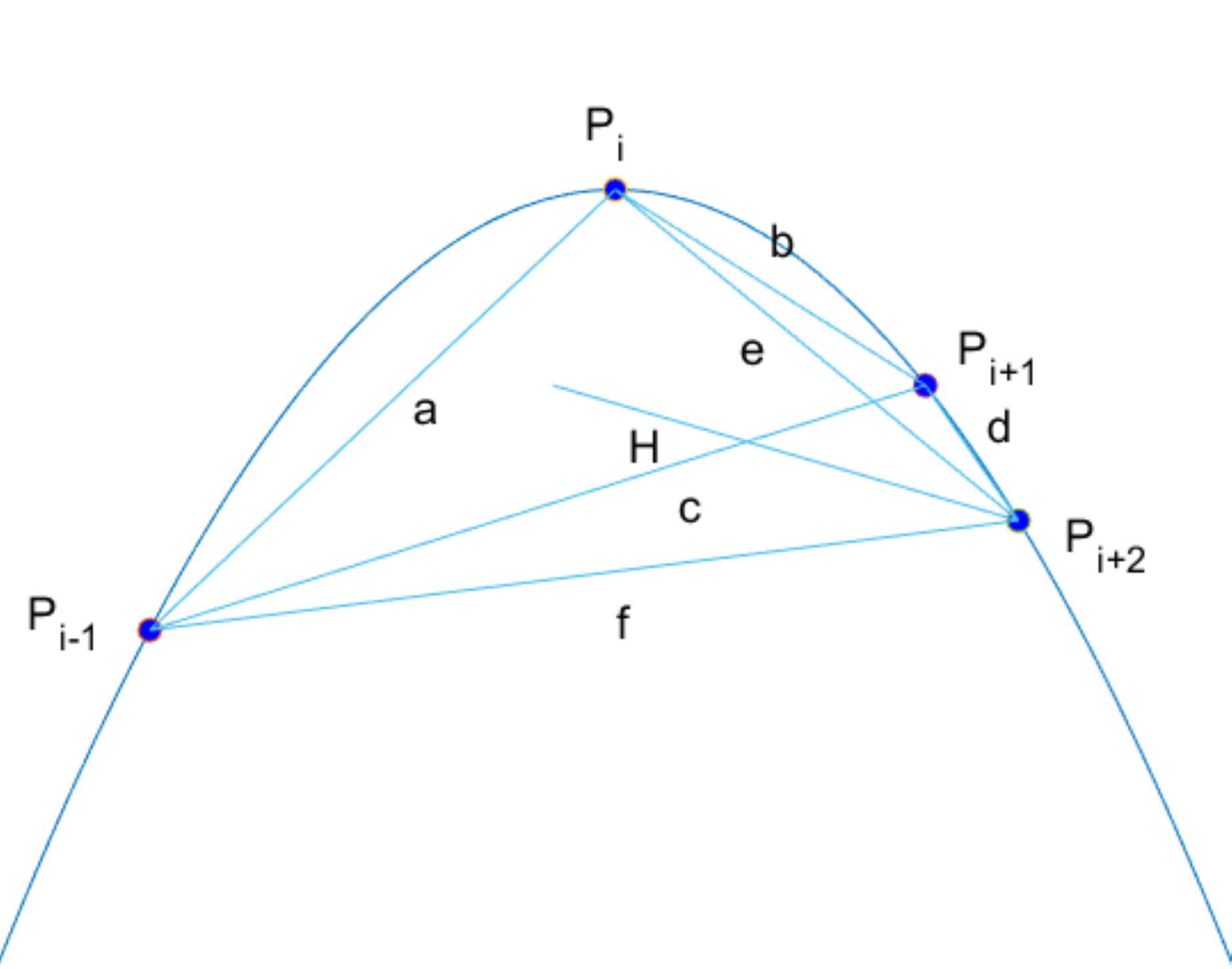}
	\caption{Consecutive points for signature approximations.}
	\label{Fig:fige2}
\end{figure}

The equations used to approximate the signature are as follows:
\begin{equation*}
	\kappa \big( {{P_i}} \big) = \displaystyle{\frac{{4\Delta }}{{abc}}}, \quad{\kappa _s}\big( {{P_i}} \big) = \displaystyle{\frac{{3\big(\kappa \left( {{P_{i + 1}}} \right) - \kappa \left( {{P_i}} \right)\big)}}{{a + b + c}}}, \quad \tau \big( {{P_i}} \big) = \displaystyle{\frac{{6H}}{{def\kappa \left( {{P_i}} \right)}}},
\end{equation*}
where $\Delta$ is the area of the triangle with sides $a, b, c$, and $H$ is the height of the tetrahedron with sides $a, b, c, d, e$, and $f$ \cite{exten}.

\subsection{Analysis of the similarity signature curve for the Lorenz attractor}\

Compared with the Euclidean signature curve, the similarity signature curve is not only invariant to translation and rotation but also invariant to scale invariance. Fig. \ref{Fig:fige3} and Fig. \ref{Fig:fige4} are the projections of the signature curve and the similarity signature curve of the Lorenz attractor on three coordinate planes. The similarity signature curve of the Lorenz attractor has a stronger self-similarity, which indirectly describes the geometric structure of the Lorenz attractor and is  shown as in Fig. \ref{Fig:fige5}.

\begin{figure}[htpb]
	\centering
	%\begin{minipage}[htbp]{1\linewidth}
	\subfigure[$x-y$ plane]{\includegraphics[height=0.47\textwidth,width=0.32\textwidth]{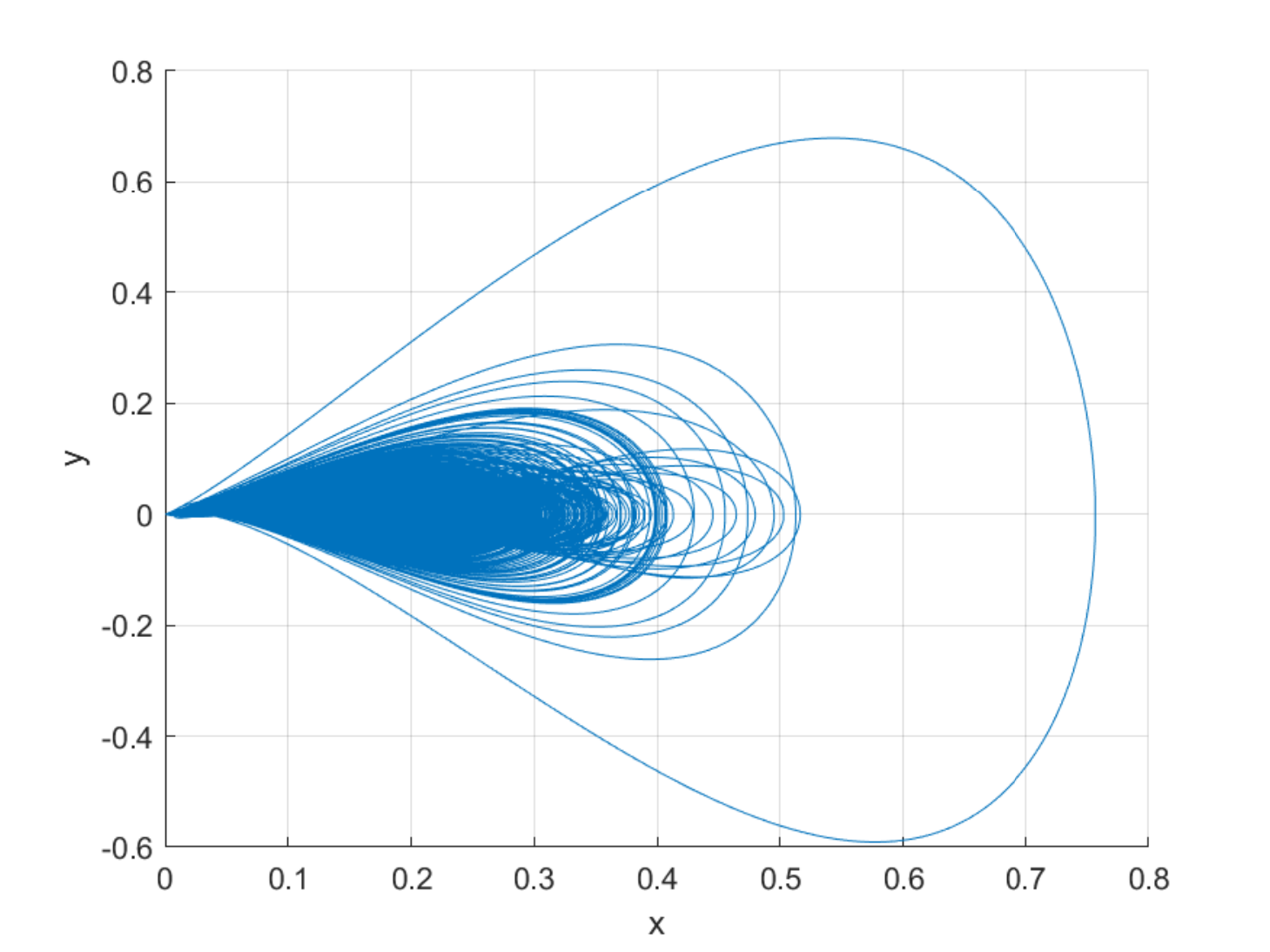}}
	\subfigure[$y-z$ plane]{\includegraphics[height=0.47\textwidth,width=0.32\textwidth]{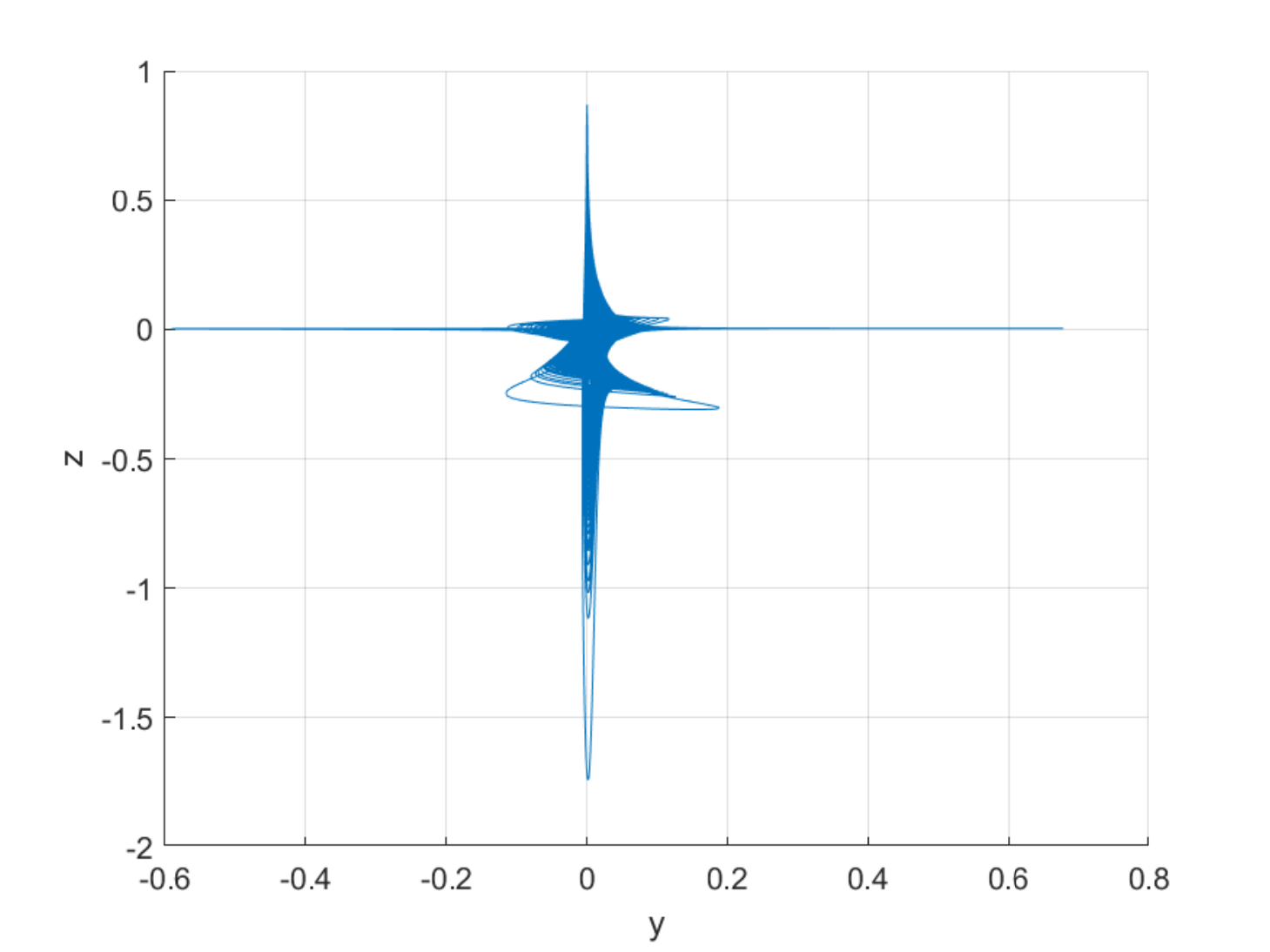}}
	\subfigure[$x-z$ plane]{\includegraphics[height=0.47\textwidth,width=0.32\textwidth]{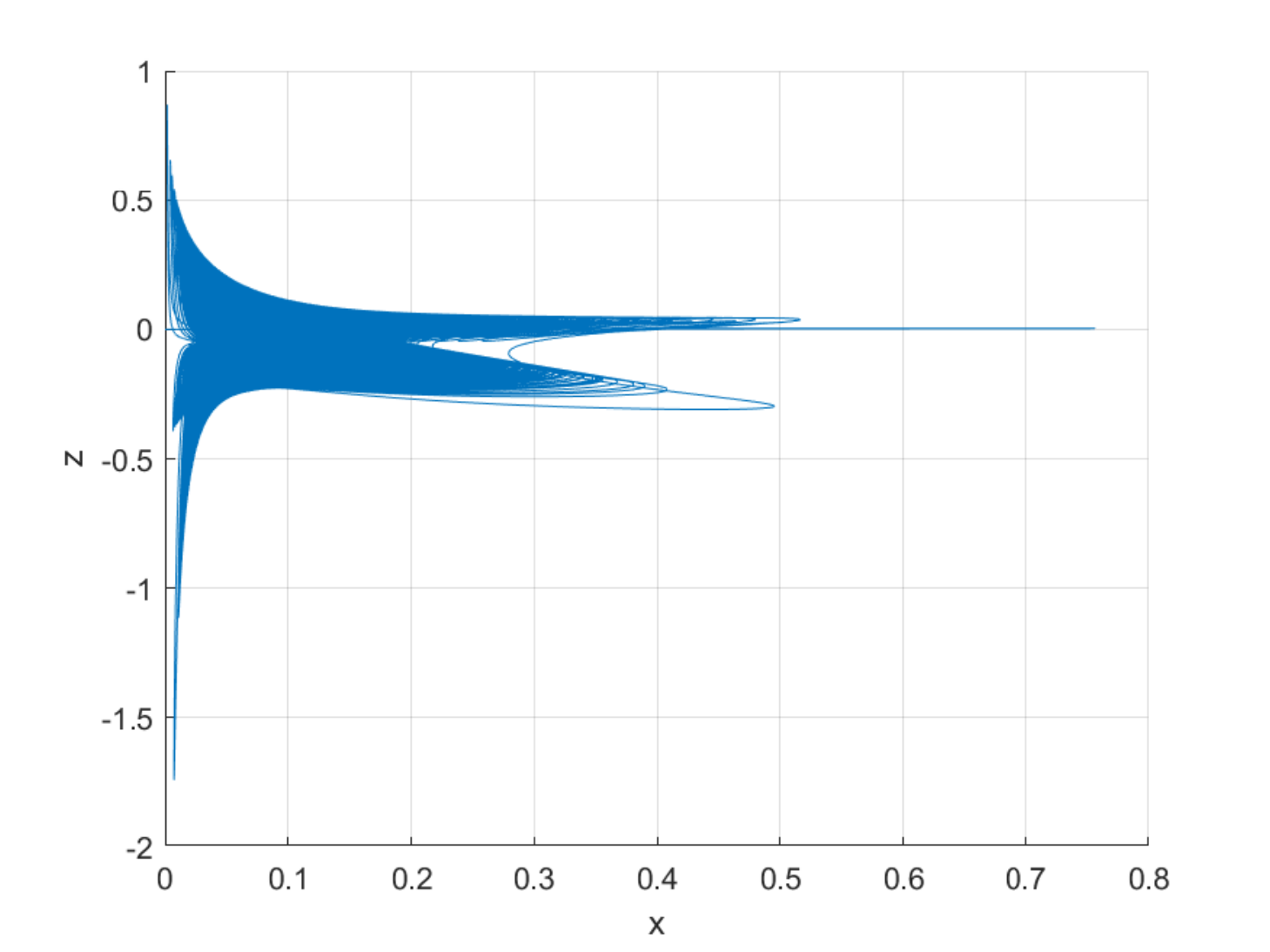}}
	\caption{Projection of signature curve of Lorenz attractor on coordinate plane.}
	\label{Fig:fige3}
	%\end{minipage}
\end{figure}
\begin{figure}[htpb]
	\centering
	%\begin{minipage}[htbp]{1\linewidth}
	\subfigure[$x-y$ plane]{\includegraphics[height=0.47\textwidth,width=0.32\textwidth]{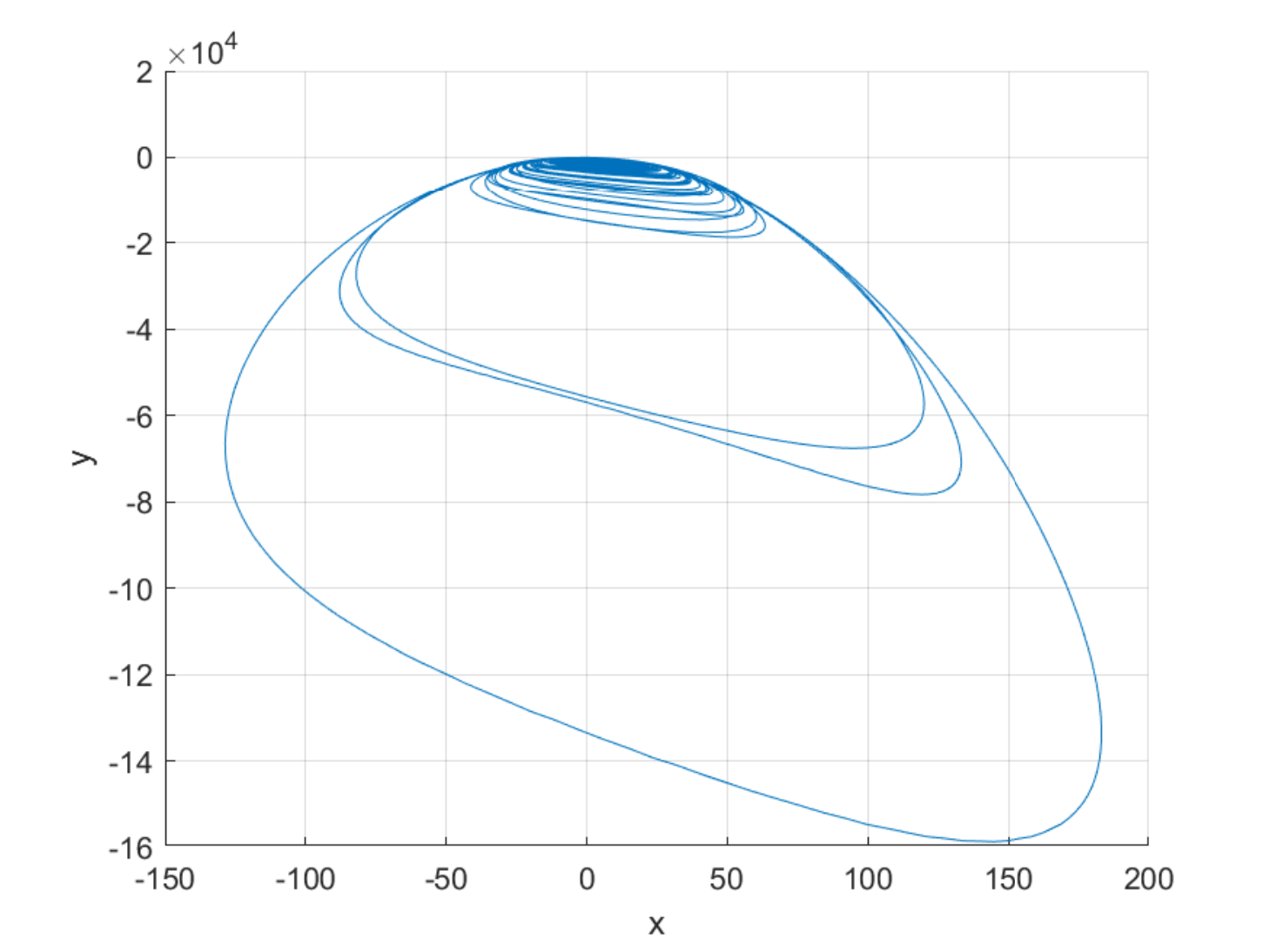}}
	\subfigure[$y-z$ plane]{\includegraphics[height=0.47\textwidth,width=0.32\textwidth]{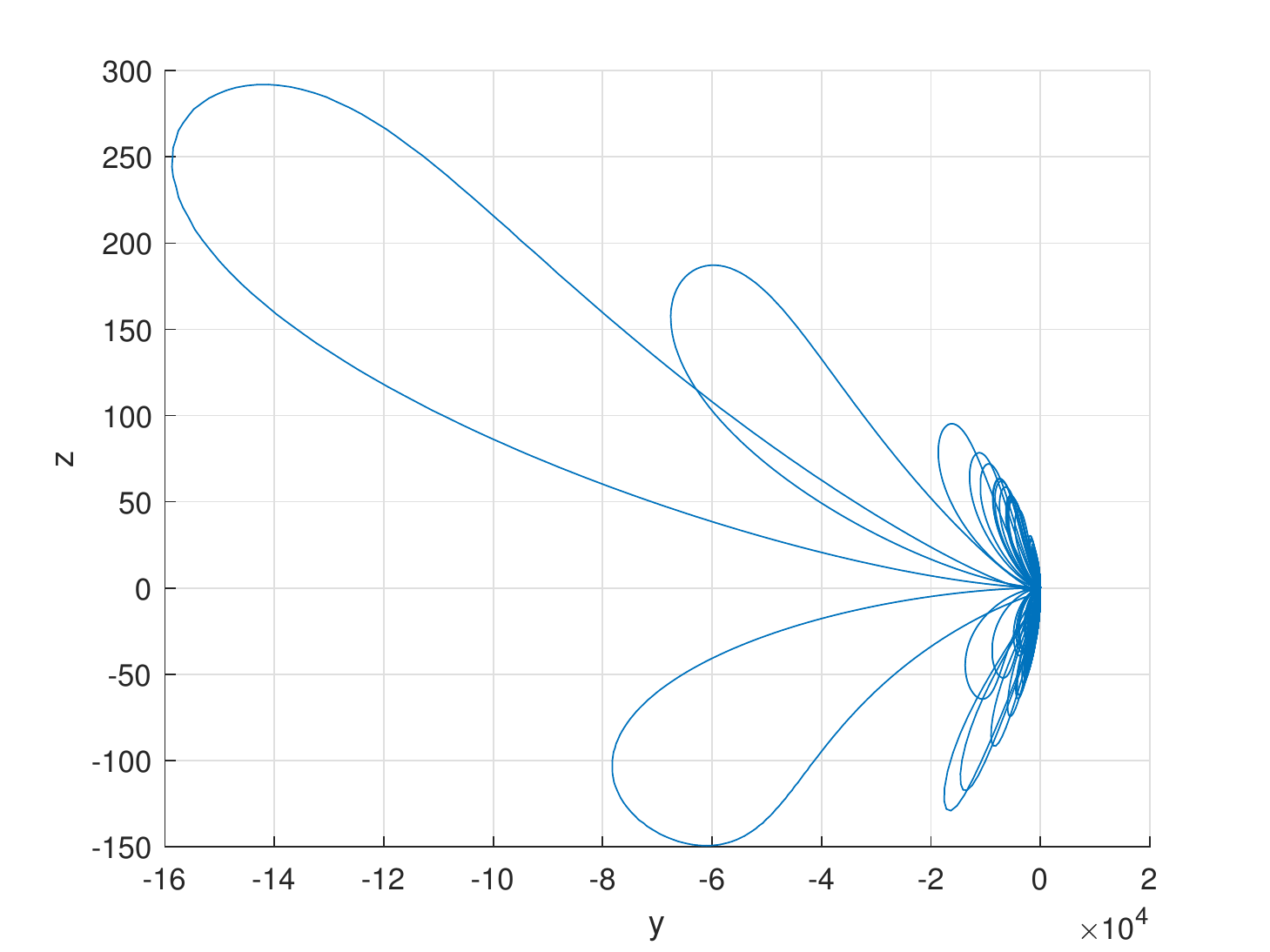}}
	\subfigure[$x-z$ plane]{\includegraphics[height=0.47\textwidth,width=0.32\textwidth]{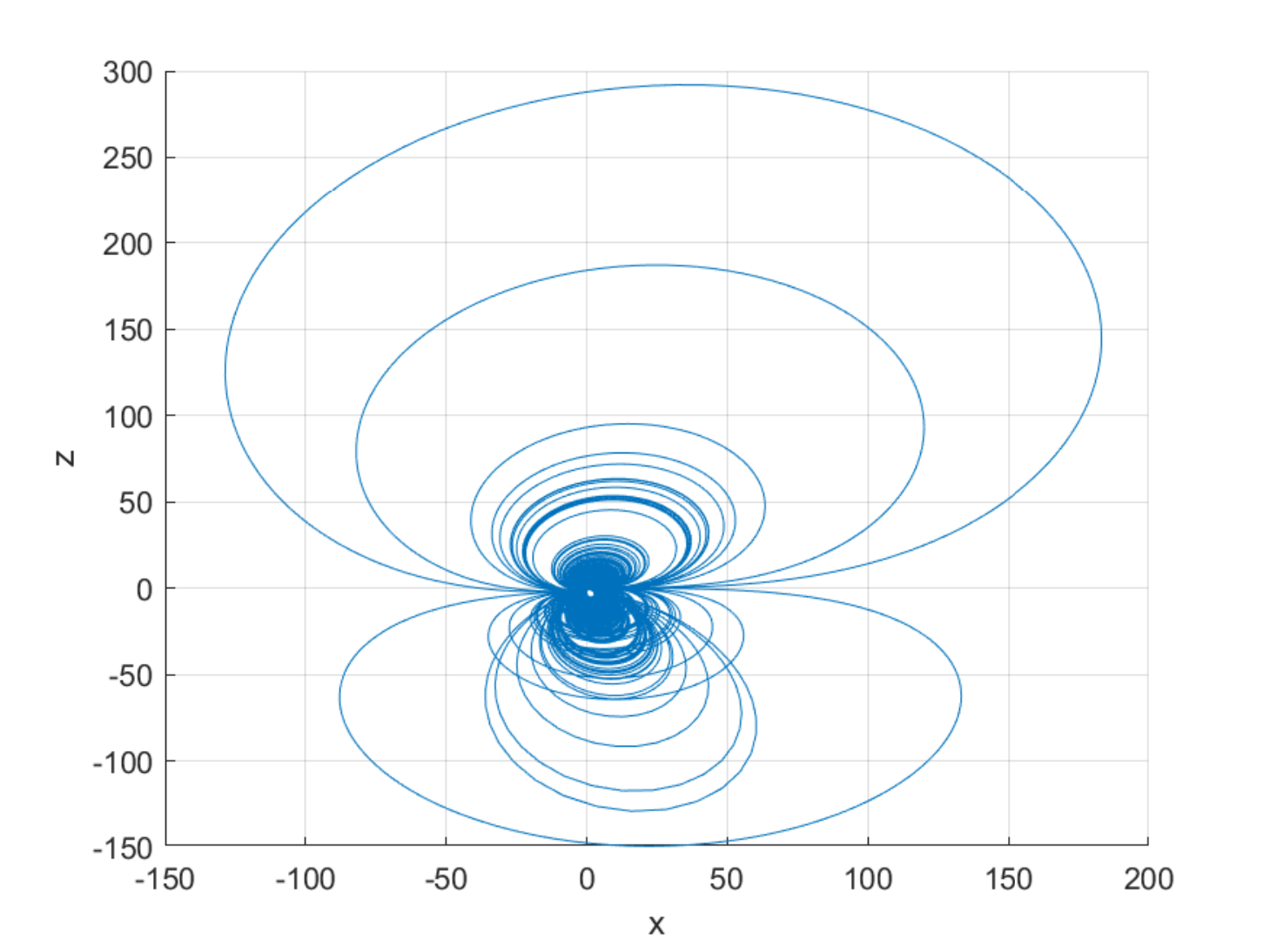}}
	\caption{Projection of similarity signature curve of Lorenz attractor on coordinate plane.}
	\label{Fig:fige4}
	%\end{minipage}
\end{figure}

%It is worth noting that if the curve in the two-dimensional space is not a $v$-regular curve, there is no guarantee that there is a one-to-one correspondence between the curve and the signature curve. Considering the counterexample in the two-dimensional space \cite{Hic} \cite{Mus}, it is obvious that the curve in the three-dimensional space will also have this problem. However, according to the evolution law of the Lorenz system, when the Lorenz system is in a chaotic state, its trajectory will not pass through the origin, but approach another stable point. Later we will further demonstrate that this problem does not arise in the Lorenz system.

It is worth noting that multiple curves with a given signature can be constructed. Hickman et al. constructed an example of a non-congruence curve with the same Euclidean signature curve \cite{Hic, Mus}. Therefore, the same consequence may occur in the similarity signature curve of the system. Later, we will use the set method to further explain that for different periodic symbol sequences, their similarity signature curves will not completely coincide.
%In \cite{exten}, any trajectory of the system is $v$-regular, and its similar signature curves do not overlap, that is to say, there is a one-to-one correspondence between the similar signature curve and the trajectory.

%\begin{figure}[htpb]
%	\centering
%	\includegraphics[height=0.45\textwidth,width=0.5\textwidth]{eps/sim_sig-eps-converted-to.pdf}
%	\caption{Similarity signature curves for Lorenz attractor.}
%\end{figure}
%\vspace{0.1cm}

Due to constraints in each direction, the intrinsic fractal structure and self-similarity of the Lorenz attractor are formed. As the decompositions of all the motion states of the Lorenz attractor, Fig. \ref{Fig:fige6} and Fig. \ref{Fig:fige7} depict two completely different states of the similarity signature curve. One is the stable state, and the other is the mutation state. The motion behavior of the Lorenz attractor can be roughly summarized as the random switching of its trajectory between two states. When the similarity signature curve is in a stable state, its corresponding trajectory will continuously rotate around a non-zero stable point. When the state changes, its corresponding trajectory will transition to another non-zero stable point.
\begin{figure}[H]
	\centering
	\includegraphics[height=0.47\textwidth,width=1\textwidth]{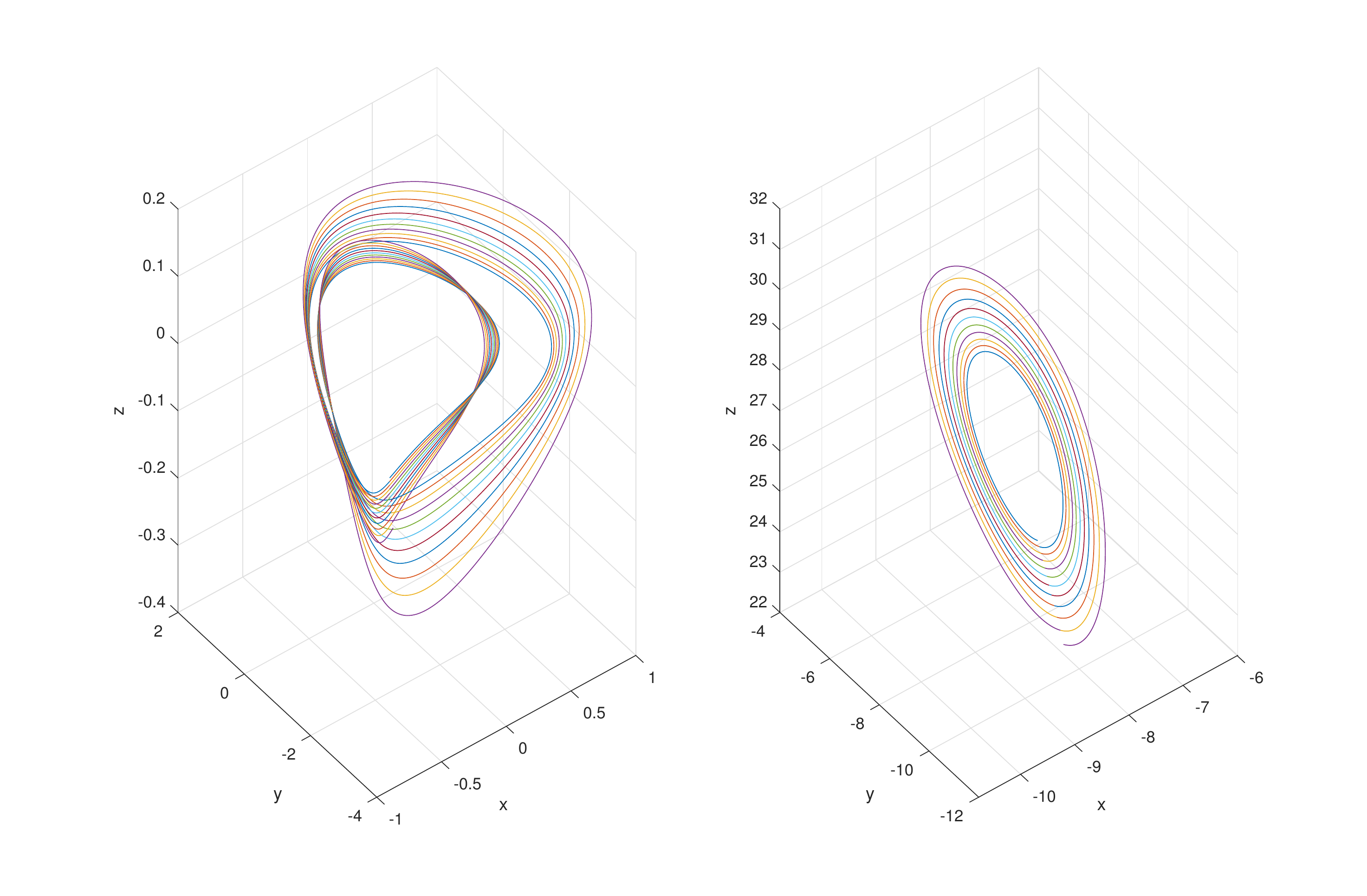}
	\caption{The Lorenz attractor with self similarity and its similarity signature curve}
	\label{Fig:fige5}
\end{figure}
%\vspace{0.4cm}

%The signature curve of the Lorenz system can also see two different states: stable state and mutation state. But The line between these two states is blurred, and even sometimes the curve trajectory of the stable state is farther than the curve trajectory of the mutation state. Therefore, it is reasonable to choose similar signature curves to analyze the Lorenz system and complete the algorithm.

%In order to verify the universality of this rule in simple chaotic systems, {\color{red} we }calculate the similar signature curve of Chen's system and obtain the similar states of the similar signature curve of Lorenz attractor. Through the evolution mode of similarity signature curve, it can be seen intuitively that Chen's system is similar but not isomorphic to Lorenz system.

\section{Method of finding periodic orbits}
The existence of periodic orbits in the Lorenz attractor has been strictly proved  \cite{exi}. On the basis that the trajectories of the Lorenz system never intersect, we propose a theorem and an algorithm to illustrate the feasibility of using the similarity signature curve to find periodic orbits. The algorithm combines the sliding window method with the similarity signature curve and uses the Krawczy operator to verify whether there is a real periodic orbit near a quasi-periodic orbit.

\subsection{Interval methods}\

Interval methods provide simple computational tests for the uniqueness, existence, and nonexistence of zeros of a map within a given interval vector \cite{qujian}. Simplification of continuous-time systems to discrete systems benefits from the Poincar\'e map. The Poincar\'e map reflects the dependence of the successor point on the predecessor point when the trajectory repeatedly crosses the same plane which can simplify the quasi-periodic orbit into a series of discrete points. The Poincar\'e map for the autonomous system is as follows.

\begin{defn}[\cite{bol}] Let ${x^*}$ be a point on a limit cycle and let $\Pi  $ be an $\left( {n - 1} \right)$-dimensional hyperplane transversal to $\Gamma $ at ${x^*}$. The	trajectory emanating from ${x^*}$ will hit $\Pi $ at ${x^*}$ in $T$ seconds, where $T$ is the minimum period of the limit cycle. Due to the continuity of ${\phi _t}$ with respect to the initial condition, trajectories starting on $\Pi $ in a sufficiently small neighborhood of ${x^*}$ will, in approximately $T$ seconds, intersect $\Pi $ in the vicinity of ${x^*}$. Hence, ${\phi _t}$ and $\Pi $ define
a mapping ${P_A}$ of some neighborhood $U \subset \Pi  $ of ${x^*}$ onto another neighborhood $V \subset \Pi  $ of ${x^*}$. ${P_A}$ is a Poincar\'e map of the autonomous system.

\begin{figure}[htpb]
\centering
%\begin{minipage}[htbp]{1\linewidth}
%	\subfigure[state 1]{\includegraphics[height=0.33\textwidth,width=1\textwidth]{eps/state1-eps-converted-to.pdf}}
%	\subfigure[state 2]{\includegraphics[height=0.33\textwidth,width=1\textwidth]{eps/state2-eps-converted-to.pdf}}
\includegraphics[height=0.47\textwidth,width=1\textwidth]{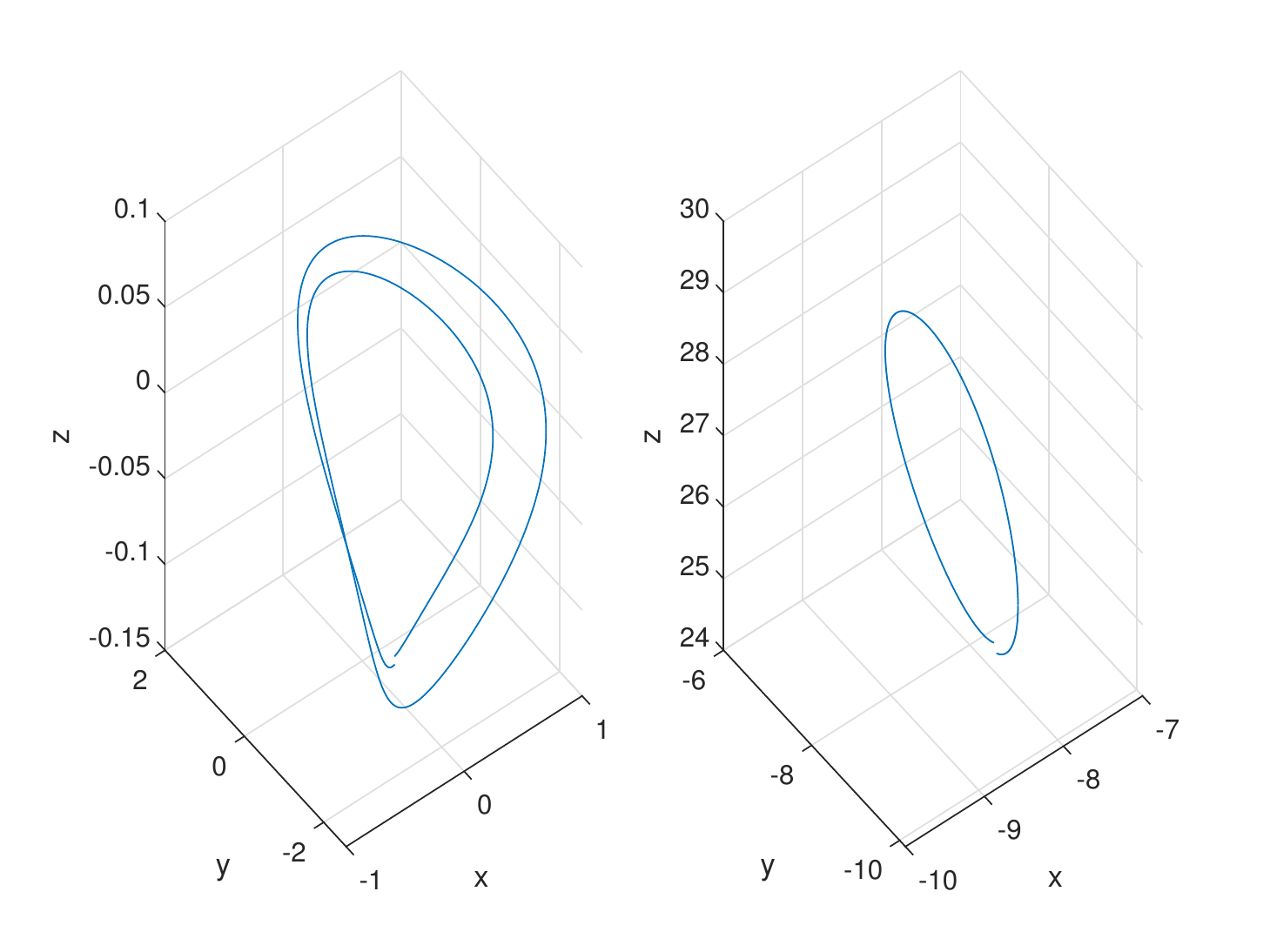}
\caption{Similarity signature curve in stable state and the trajectory of its corresponding the Lorenz attractor.}
\label{Fig:fige6}
%\end{minipage}
\end{figure}

\begin{figure}[htpb]
\centering
%	\begin{minipage}[htbp]{1\linewidth}
\includegraphics[height=0.47\textwidth,width=1\textwidth]{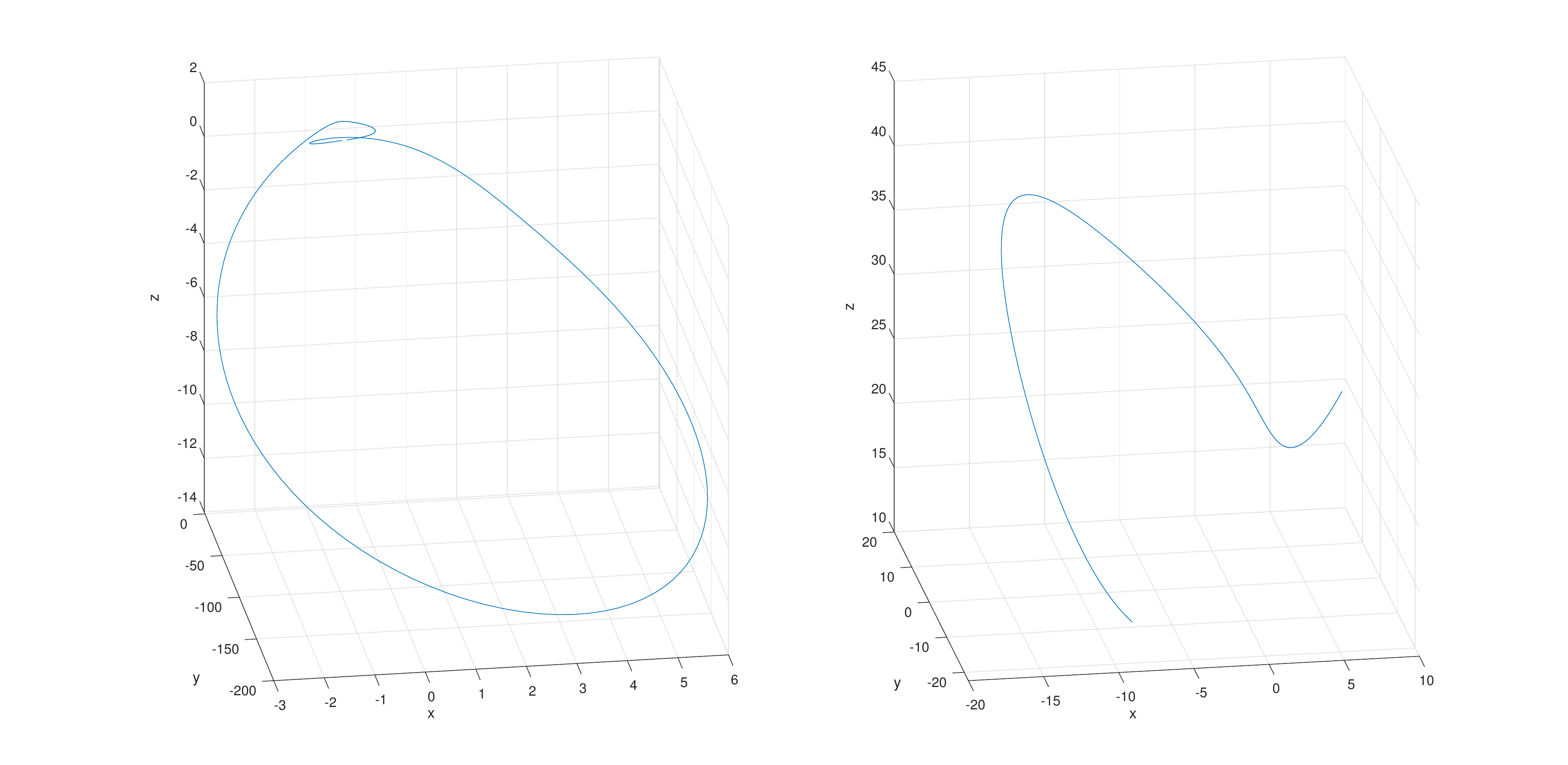}
\caption{Similarity signature curve in mutation state and the trajectory of its corresponding the Lorenz attractor.}
\label{Fig:fige7}
%\end{minipage}
\end{figure}

%	For a flow $\phi \left( {{x^0};t} \right)$ in space, there exists a hyperplane $\varphi $ such that it contains ${x^0}$ but is not tangent to $\phi \left( {{x^0};t} \right)$. It can be seen that ${x^0}$ is the first intersection point of flow $\phi \left( {{x^0};t} \right)$ and $\varphi $. Motion along the flow, if there is a time ${t_1} > 0$ such that $\phi \left( {{x^0};{t_1}} \right) \in \varphi $. Then {\color{red} we }say that $\phi \left( {{x^0};0} \right)$ is mapped to $\phi \left( {{x^0};t_1} \right)$ by Poincar\'e map and call that:
%	\begin{center}
%		$P\left( {\phi \left( {{x^0};{t_1}} \right)} \right) = \phi \left( {{x^0};{t_1}} \right)$
%	\end{center}
\end{defn}
The Poincar\'e map can be simply called as $\Omega :\Pi  \to \Pi  $ and it is defined as $\Omega\left( x \right) = \varphi \left( {t \left( x \right),x} \right)$ where $\Pi   = {\Pi  _1} \cup  \cdots \cup {\Pi  _m}$ is the union of hyperplanes, and $t \left( x \right)$ is the return time after which the trajectory $\varphi \left( {t,x} \right)$ returns to $ \Pi  $. Periodic points of $ \Omega $ correspond to periodic orbits of the continuous system \cite{qujian}.

%According to the Poincar\'e map, {\color{red} we }construct Poincar\'e-like map:
%
%\begin{defn}
%	For a flow $\phi \left( {{x^0};t} \right)$ in space, there is a sequence $S$ of segmentation points such that it contains ${x^0}$. It can be seen that ${x^0}$ is the first point in sequence $ S $. Motion along the flow, if there is a time ${t_1} > 0$ such that $\phi \left( {{x^0};{t_1}} \right) \in S $. Then {\color{red} we }say that $\phi \left( {{x^0};0} \right)$ is mapped to $\phi \left( {{x^0};t_1} \right)$ by the map and call that:
%	\begin{center}
%		$P\left( {\phi \left( {{x^0};{t_1}} \right)} \right) = \phi \left( {{x^0};{t_1}} \right)$
%	\end{center}
%\end{defn}
%
%{\color{red}We} can also simply call the Poincar\'e-like map $f: S \to S$ is define as $f\left( x \right) = \varphi \left( {x,\tau \left( x \right)} \right)$, where $S$ is the trajectory of the continuous system and $\tau \left( x \right)$ is the movement time of the trajectory starting from point $x$. Periodic points of $f$ correspond to periodic orbits of the continuous system.

A map is constructed to study the existence of period-$p$ orbits of $\Omega$:
\begin{equation*}
{\left[ {F\left( z \right)} \right]_k} = {x^{\left( {\left( {k + 1} \right)\bmod p} \right)}} - \Omega\left( {{x^k}} \right),
\end{equation*}
where $0 \le k \le p$ and $z = \left( {{x^{\left( 0 \right)}},{x^{\left( 1 \right)}}, \cdots ,{x^{\left( {p - 1} \right)}}} \right)$. Zeros of $F$ correspond to period-$p$ points of $\Omega$, i.e., $F\left( z \right) = 0$ if and only if ${\Omega^p}\left( {{x^{\left( 0 \right)}}} \right) = {x^{\left( 0 \right)}}$.

The interval method can test the uniqueness, existence, and non-existence of the zeros of the Poincar\'e map, so it can be used to find the periodic solution of the Lorenz system on the Poincar\'e section. The Krawczyk operator \cite{Kra,Neu} is:
\begin{equation*}
K\left( \mathbf{z} \right) = \hat z - MF\left( {\hat z} \right) - \left( {MF'\left( \mathbf{z} \right) - I} \right)\left( {\mathbf{z} - \hat z} \right),
\end{equation*}
where $\hat z \in \mathbf{z}$, $F\left( {\hat z} \right)$ is the interval matrix containing the Jacobian matrices $F\left( {\hat z} \right)$ for all $ z \in \mathbf{z}$, and $M$ is an invertible matrix. Using the method given in  \cite{flo}, the map $\Omega\left( x \right)$ and its Jacobian matrix in the Krawczyk operator can be obtained. Generally, we take $\hat{z}$ as the center of $\mathbf{z}$ and ${M}$ as the inverse of $F'\left( z \right)$. If $ \mathbf{z} \cap K\left( \mathbf{z} \right) = \phi $, then there are no period-$p$ orbits in $ \mathbf{z} $. If $K\left( \mathbf{z} \right) \subset {\mathop{\rm int}}  \mathbf{z}$, there exists precisely one period-$p$ orbit in $\mathbf{z}$ \cite{Phy}.

\subsection{Periodic orbits by the symbolic dynamics based approach}\

The following theorem guarantees the feasibility of the algorithm proposed in this paper.
\begin{thm}
For two different periodic orbits $P_1$ and $P_2$, their periodic symbol sequences are ${S_1} = \left( {{s_{11}},{s_{12}}, \ldots ,{s_{1{n_1}}}} \right)$, ${S_2} = \left( {{s_{21}},{s_{22}}, \ldots ,{s_{2{n_2}}}} \right)$, $1 < {n_2} < {n_1}$, and $S_2$ is a ordered subset of $S_1$. The similarity signature curve of $P_2$ will not completely coincide with the similarity signature curve of $P_1$.
\end{thm}
\begin{proof}
For a given periodic symbol sequence $S_0 = \left( {{s_0},{s_1}, \ldots ,{s_{p - 1}}} \right)$, there exists at most one point $x \in N$ with the symbol sequence $S$ such that ${R^p}\left( x \right) = x$ \cite{tuck}. Then periodic orbits with the same sequence of periodic symbols will not be strictly similar. If the long periodic orbits do not contain short periodic orbits, it means that there will be no periodic orbit with period symbol $S_2$ in the long period orbit, so this theorem is correct. Otherwise, we will face the following two cases.

(1) $P_1$ starts with $P_2$ or ends with $P_2$, that is, ${S_1} = \left( {{S_2},{s_{{n_2} + 1}}, \cdots ,{s_{{n_1}}}} \right)$ or ${S_1} = \left( {{s_{11}}, \cdots ,{s_{{n_1} - {n_2}}},{S_2}} \right)$.  For the given $S_2$, there exists at most one point $x \in N$ with the symbol sequence $S_2$ such that ${R^{{n_2}}}\left( x \right) = x$. The point $x$ exists in the trajectory satisfying the periodic symbol sequence and belongs to the intersection of attractor and trapping region $N$,  then $x$ also satisfies ${R^{{n_1}}}\left( x \right) = x$. The part of ${S_1}$ excluding ${S_2}$ can be regarded as a whole, which is called ${S_3}$, that is, ${S_1} = \left( {{S_2},{S_3}} \right)$ or ${S_1} = \left( {{S_3},{S_2}} \right)$.  ${S_3}$ is also a sequence of periodic symbols, where ${S_3} = \left( {{s_{{n_2} + 1}}, \cdots ,{s_{{n_1}}}} \right)$ or ${S_3} = \left( {{s_{11}}, \cdots ,{s_{{n_1} - {n_2}}}} \right)$.

(2) $P_2$ is included in $P_1$, that is, ${S_1} = \left( {{s_{11}}, \cdots, {s_{1k}}, {S_2}, {k + {n_2} + 1 }, \cdots ,{s_{{n_1}}}} \right)$, and then the proof process is similar as the above case.
\end{proof}

In \cite{tuck}, it was established that there existed a finite set of rectangles $N = {N^ - } \cup {N^ + } \subset \Pi $ such that $N$ is mapped into itself by the return map $R:\Pi \to \Pi  $.  They formed the forward invariant set $\mathscr{A}= \bigcap\nolimits_{k = 0}^\infty  {{P^k}\left( N \right)} $. In fact, $\mathscr{A}$ is the intersection of the Lorenz attractor with $N$. For a trajectory $\left\{ {{x_n}} \right\}_{n = 0}^\infty  \subset N$ where ${x_{n + 1}} = R\left( {{x_n}} \right)$.
Suppose that $x'$ is the similarity signature of $x$, and we need at least 4 points to approximate $x'$. The points in the Lorenz attractor will not intersect. According to the continuity of  the Lorenz system, we know that the similarity signature curve of the Lorenz attractor is also continuous. When point $x$ satisfies the condition $R{^p}\left( x \right) = x$, the similarity signature of $x$ is also invariant, so there is an invariant set in the similarity signature of the Lorenz attractor. The similarity signature at this point also exists and is unique, otherwise, there will be conflicts.

\subsection{Sliding window method based on similarity signature curve}\

The sliding window method can solve the problems of finding the properties, such as the length, of a continuous interval that satisfies certain conditions. All trajectories of the Lorenz attractor are processed successively by its similarity signature curve and sliding window method, and the segmented fragments have obvious self-similarity, which is a manifestation of the fractal structure of the Lorenz attractor. It is wise to construct a real periodic orbit based on a quasi-periodic orbit but the initial value of the system limits its accuracy.  Experiments show that in the process of finding quasi-periodic orbits, the choice of initial point will affect the final result. A suitable initial point can ensure that the quasi-periodic orbits are similarly positioned in the Lorenz system. In addition to the initial point, the appropriate window size can more intuitively reflect the relationship between the properties of the data in each window, which is helpful for the division of periodic orbits. The algorithm for finding the periodic orbits of the Lorenz system is implemented according to the rules proposed above.

The algorithm simplifies the problem of finding the periodic orbits of the Lorenz system into the search and combination of two different states of the trajectory. For purposes of clarity, we generate the following pseudo-code and algorithm structure.

\begin{algorithm}
\renewcommand{\algorithmicrequire}{\textbf{Input:}}
\renewcommand{\algorithmicensure}{\textbf{Output:}}
\caption{Sliding window method}
%	\label{alg1}
\begin{algorithmic}[1]
\REQUIRE The Lorenz system dataset $D$; Window size $W$; Initial point ${P_{1,1}}$
\ENSURE Periodic orbit segmentation points $z$
\STATE Initialize $z = \left[ {{P_{1,1}}} \right]$, $poi = 0$, $i = 1$;
\STATE Compute $\tilde \kappa$, ${\tilde \kappa _{\tilde s}}$ and $\tilde \tau$ of the Lorenz system;
\WHILE{$poi + W \leqslant length\left( D \right)$}
\FOR{$j = 1 : W$}
\STATE Compute $d\left( j \right) = dist\left( {{P_{i,1}},{P_{i,j}}} \right) $;
\STATE Compute $\left[ {val,index} \right] = \min d\left( j \right)$;
\STATE Compute ${P_{i + 1,1}} = {P_{i,index}}$;
\STATE $z = \left[ {z,{P_{i + 1,1}}} \right]$;
\ENDFOR
\STATE $poi = poi + index$;
\STATE $i = i + 1$;
\ENDWHILE
\STATE \textbf{return} $z$;
\end{algorithmic}
\end{algorithm}

The similarity signature curve for the Lorenz systems is acquired through the numerical calculation by using the formulas in Sec. 2. Actually, a proper selections of the initial point and window size for the algorithm can ensure the feasibility of the algorithm. The initial point should satisfy two conditions. (1) The closest point to the initial value of the system. (2)  The extreme point of the similarity signature curve in the $z$-axis direction.
%Considering the above two constraints are necessary, we choose the 2000th point of the discrete Lorenz system.
These ensure that the segmentation points on the trajectory of the Lorenz attractor are more stable. When the trajectory transitions from rotating around one stable point to another stable point, the period of its similarity signature curve is the longest, so the maximum periodic orbit length of the system should be taken into account in the selection window size. If the window size is too small, a complete period may not exist in the window. Suppose the window size is $W$ which satisfies the condition slightly smaller than the length of two periods of the similarity signature curve. Let $W=2500$, equates to about $1.25s$. If the length of the last remaining data is less than the size of the window, it will not be considered. Suppose that the initial point of each window is ${p_{i1}}\left( {1 \leqslant i \leqslant n} \right)$, and the other points in a window period are ${p_{ij}}\left( {2 \leqslant j \leqslant 2500} \right)$, where ${n}$ is equal to the number of windows required in the end. By calculating the distances between all points in the longer trajectory and the fixed initial point, we observe the existence of cycles in their distance image, as shown in Fig. \ref{Fig:fige8}. After completing the closest point search for the initial points of the similarity signature curve, the trajectory between these two points is always in the form of two circle-like nests. As the degree of the mutation of the similarity signature curve gradually increases, the amplitude of the distance map also increases, which means that the Lorenz attractor will transition from a stable state to a mutation state.
\begin{figure}[htpb]
\centering
%\begin{minipage}[htbp]{1\linewidth}
\subfigure{\includegraphics[height=0.47\textwidth,width=0.45\textwidth]{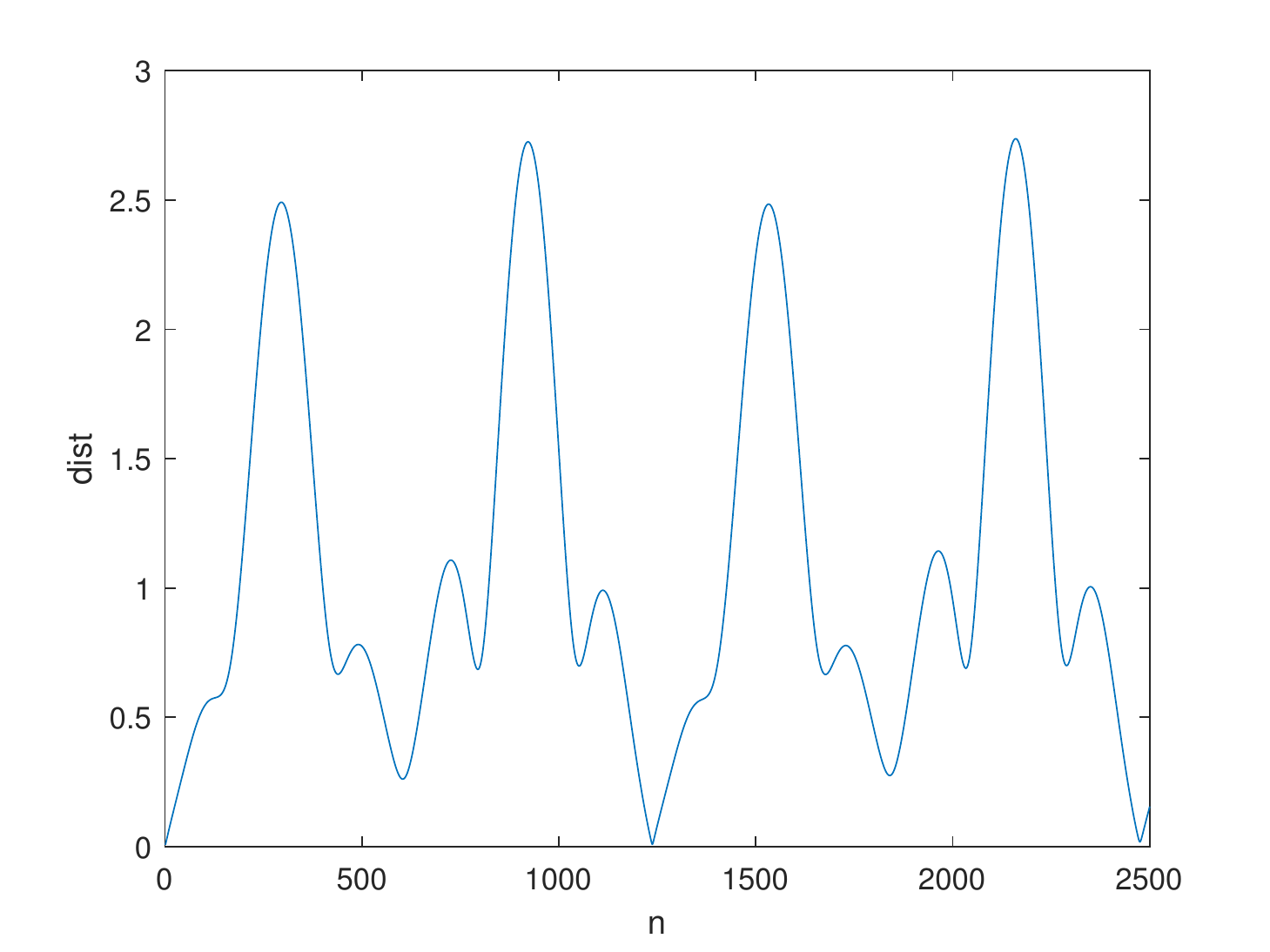}}
\subfigure{\includegraphics[height=0.47\textwidth,width=0.45\textwidth]{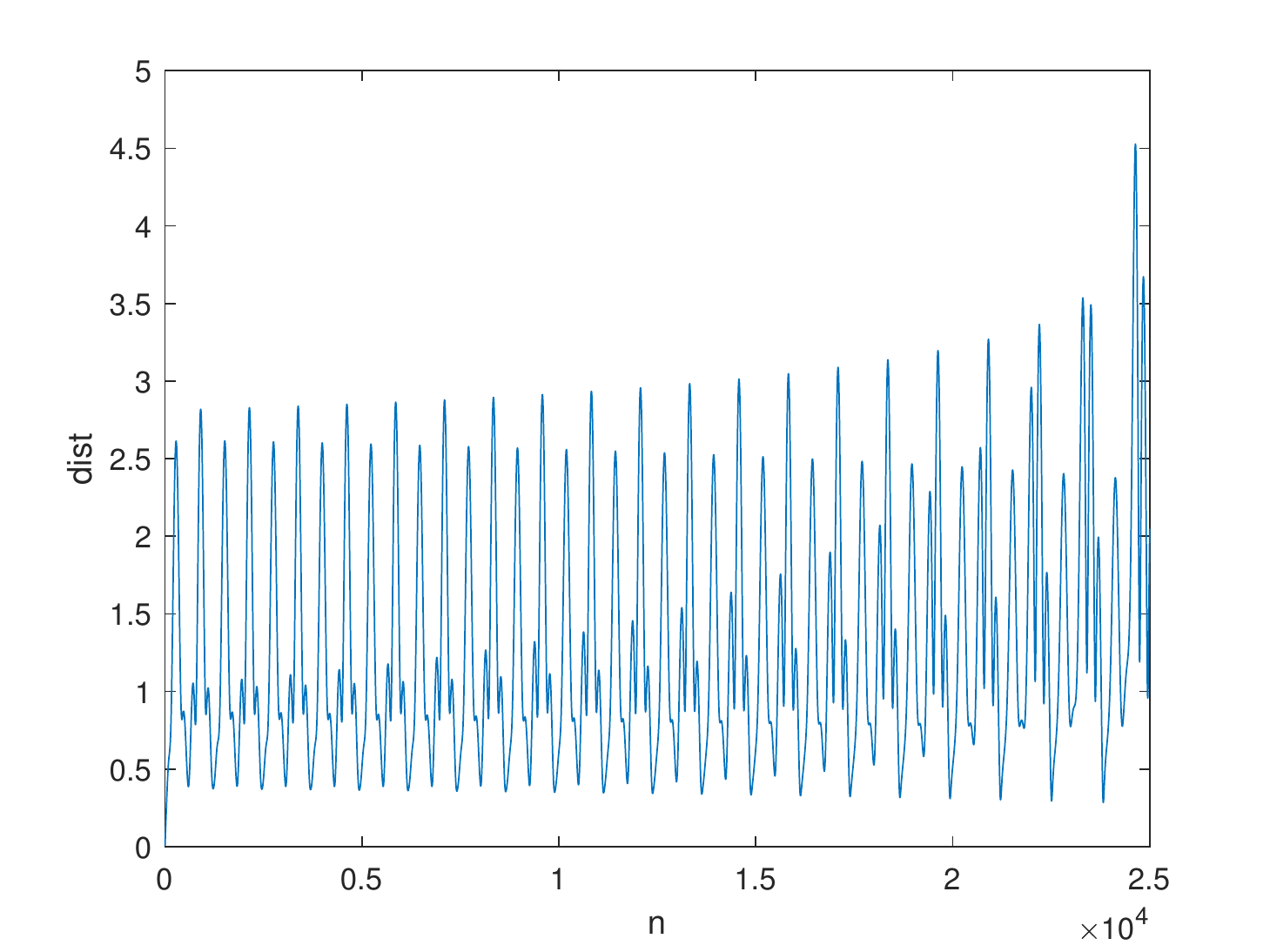}}
\caption{Distance between all points in a window and the initial point.}
%\end{minipage}
\label{Fig:fige8}
\end{figure}

Second, as the trajectory continues to rotate around a non-zero stable point, the point closest to the initial point on each approximate circle in its similarity signature curve gradually moves away from the initial point. As shown in Fig. \ref{Fig:fige9}, the distance between them gradually increases. When the point ${p_{ik}}$ closest to the initial point ${p_{i1}}$ is found, the initial point needs to be updated. The point ${p_{ik}}$ will be called the new initial point, and then we recalculate the distance image and update the initial point continuously until the subsequent trajectory ends. Then point ${p_{ik}}$ is equal to point ${p_{i+1,1}}$. This part is a detailed explanation of the loop part in the pseudocode.
\begin{figure}[htpb]
\centering
%	\begin{minipage}[htbp]{1\linewidth}
\includegraphics[height=0.47\textwidth,width=0.5\textwidth]{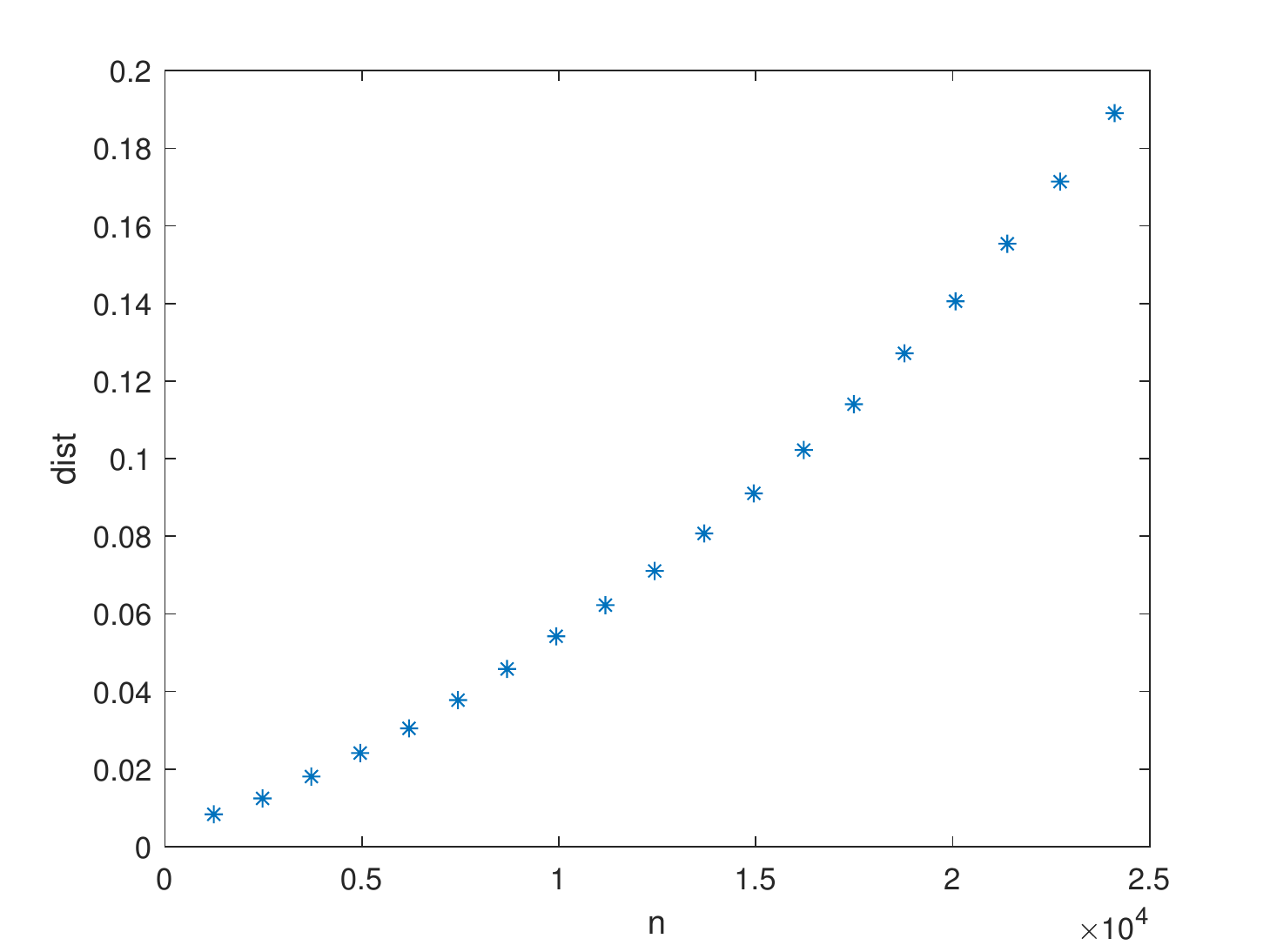}
\caption{The points in the graph represent the distance between the points closest to the starting point in each cycle.}
\label{Fig:fige9}
%\end{minipage}
\end{figure}

Third, a sequence of points $p = \left\{ {{p_{11}},{p_{21}}, \cdots ,{p_{n1}}} \right\}$ is obtained by looping, where $n$ is the number of segmentation points and  windows required. Filtering segmentation points leaves only points along the same direction of the $z$-axis. The filtered points form a new point sequence, which is assumed to be $P = \left\{ {{P_{11}},{P_{21}},...,{P_{k1}}} \right\}$, where point ${P_{i1}}$ is equivalent to point ${p_{k1}} \in p$. The trajectory between every two points in the point sequence $P$ may be a quasi-periodic orbit. Except for the first and last points in the sequence, the other points are the end of one orbit and the beginning of the next. The endpoints of the quasi-periodic orbits obtained by applying the similarity signature curve can be kept within the neighborhood of 0.0005.

Finally, the periodic orbits are closed. This method is based on the existence of a real periodic orbit in a neighborhood of the quasi-periodic orbit \cite{close}. Table \ref{table1} gives the data for quasi-periodic orbits with periods not greater than 5.
%The results concerning quasi-periodic orbits with period $p \leqslant 5$ are collected in Table \ref{table1}.
%For each quasi-periodic orbit its period $p$, the distance $Dist$ between endpoints and the corresponding symbol sequence are reported.

In short, there are three advantages of using the similarity signature curve to find periodic orbits: (1) It can reflect the self-similarity of the Lorenz system from the perspective of invariants; (2) High precision is achieved by common methods, which is helpful for quasi-periodic orbits positioning; (3) The order of the periodic orbits is guaranteed and no omission occurs.

\begin{table}[htpb]
\centering
\setlength{\abovecaptionskip}{0pt}
\setlength{\belowcaptionskip}{10pt}
\caption{Data of periodic orbits with period $p \leqslant 5$, including the coordinates of the segmentation points and the distance between the endpoints.}
\renewcommand{\arraystretch}{1.4} % 调整行间距数字越大，行间距越大
\setlength{\tabcolsep}{0.25cm}{  % 2cm表示列宽
\begin{tabular}[htbp]{lllll}
	\hline
	$s$ & $x$ & $y$ & $z$ & dist \\
	\hline
	LR & -9.727902750 & -16.398960521 & 17.292212679 & 0.002599810\\
	LLR & -9.836367146 & -16.255530452 & 18.121784817 & 0.002386087\\
	LLLR & -9.835075562 & -16.131007675 & 18.411435859 & 0.003081481\\
	LLRR & -9.917578308 & -15.670162223 & 19.764832024 & 0.003614604\\
	LLLLR & -9.918073021 & -15.155786811 & 20.761179930 & 0.001839419\\
	LLLRR & -9.849285414 & -15.553696219 & 19.723034119 & 0.002946679\\
	LLRLR & -9.803205079 & -16.451041061 & 17.539781868 & 0.003687509\\ \hline
\end{tabular}}
\label{table1}
\end{table}

\subsection{Method summary}\

In this method, quasi-periodic orbits are obtained by dividing and recombining the Lorenz attractor by its similarity signature curve. A better approximation of the position of the periodic orbit near the quasi-periodic orbit is found by standard (non-interval) Newton iteration. Finally, the existence of real periodic orbits can be proved by using the Krawczyk operator.

\section{Periodic orbits for the Lorenz system}
According to the algorithm proposed in the previous section, a sequence $P = \left( {{P_{11}},{P_{21}}, \cdots ,{P_{k1}}} \right)$ is obtained, where $P_{i1}$ is the point in the Lorenz attractor and is also the starting and ending points of orbits. The standard Newton method contributes to sharpening the approximation and thus obtaining quasi-periodic trajectories.  Finally, the Krawczyk operator is used to prove the existence of real periodic orbits in the neighborhood of their quasi-periodic orbits and all periodic orbits with period $p \leqslant 8$ are found. The orange part in Fig. \ref{Fig:fige10} is composed of the segmentation points of each periodic orbit. The positions of the segmentation points exhibit similarity due to the invariance to scaling of the similarity signature curve. Each trajectory is labeled in the following ways. Let $q$ be the first intersection of the two-dimensional flow and the return plane $\Pi $. When the trajectory intersects the plane to the left of $q$, it is recorded as L, otherwise, recorded as R. Period-$p$ orbits can be better represented using the sign sequence $S = \left( {{s_1},{s_2}, \cdots ,{s_p}} \right)$, for ${s_i} = L$ or ${s_i} = R$, for $i = 1,2, \cdots ,p$.
\begin{figure}[htpb]
\centering
%\begin{minipage}[htbp]{1\linewidth}
\includegraphics[height=0.47\textwidth,width=1\textwidth]{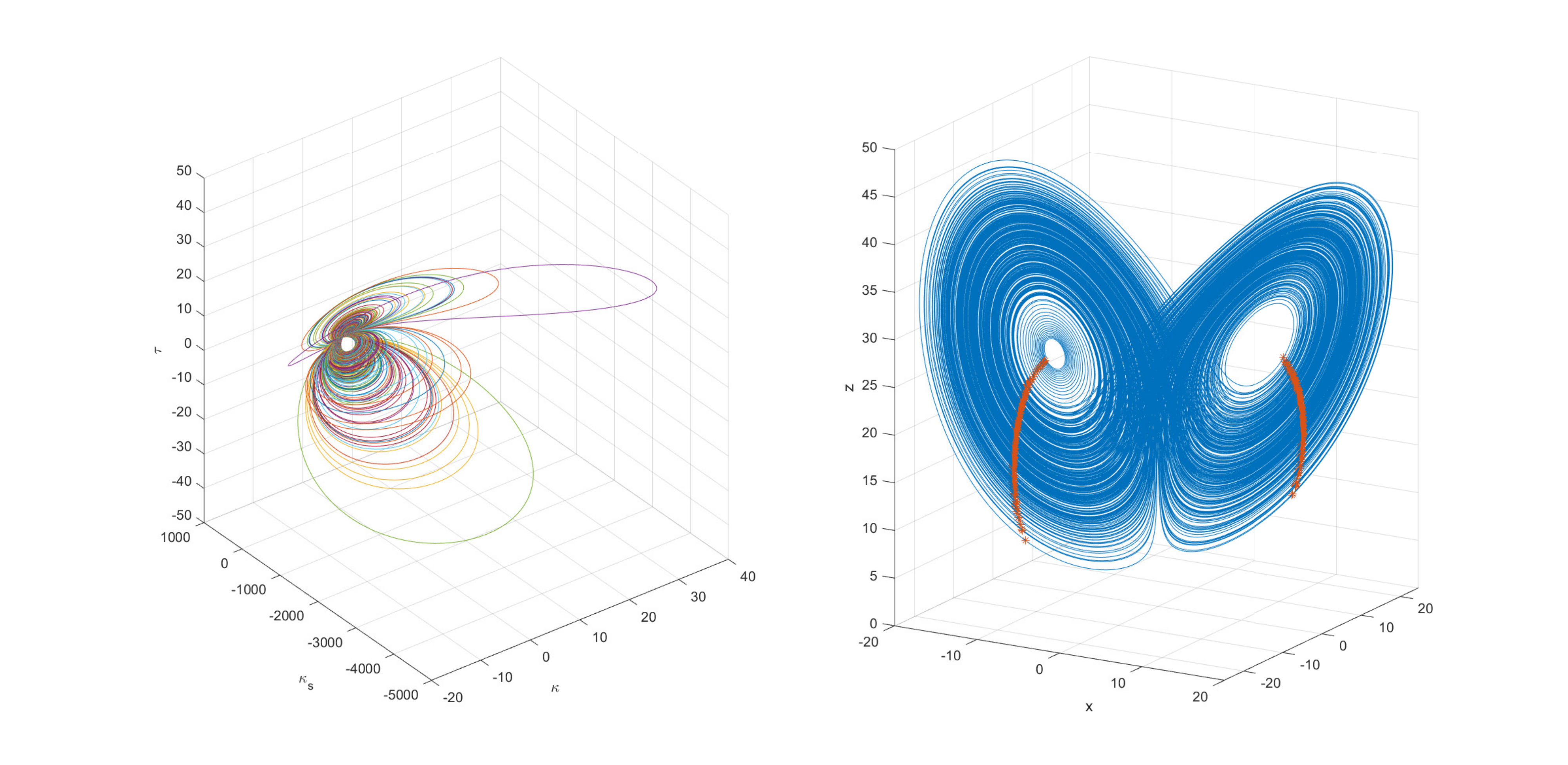}
\caption{Divided similarity signature curve and the positions of the segmentation points on the Lorenz attractor. }
\label{Fig:fige10}
%\end{minipage}
\end{figure}

There are infinitely long periodic orbits in the Lorenz system, and these long-period orbits can be composed of multiple short-period orbits. Only the coprime periodic orbits are considered when recording periodic orbits. If a sequence is a periodic orbit, only its non-repetitive serial number is recorded as the equivalent form of the periodic orbit.

The one-dimensional symbolic dynamics of the Lorenz system are established to show all the found periodic orbits with periods no greater than 8. We give the period $p$, the period length $T$, and the corresponding symbol sequence $s$ for each periodic orbit discovered. The results are reported in Table \ref{table2}. Fig. \ref{Fig:fige11} illustrates the periodic orbits with period $p \leqslant 6$. For each pair of symmetric orbits, only periodic orbits with a greater number of L than R are plotted.

\begin{table}[htpb]
\centering
\setlength{\abovecaptionskip}{0pt}
\setlength{\belowcaptionskip}{10pt}
\caption{Short periodic orbits, where $p$ is the period of the orbit, $T$ is the flow-time, $s$ is the corresponding symbol sequence.}
\renewcommand{\arraystretch}{1.2} % 调整行间距数字越大，行间距越大
\setlength{\tabcolsep}{0.3cm}{  % 2cm表示列宽
	\begin{tabular}[htbp]{llllll}
		\hline
		$p$ & ${T}$ & $s$ & $p$ & ${T}$ & $s$\\
		\hline
		2 & 1.55865 & LR & 7 & 5.39421 & LLRLLRR \\
		3 & 2.30591 & LLR & 7 & 5.42912 & LLRLRLR \\
		4 & 3.02358 & LLLR & 8 & 5.78341 & LLLLLLLR \\
		4 & 3.08428 & LLRR & 8 & 5.92499 & LLLLLLRR \\
		5 & 3.72564 & LLLLR & 8 & 5.99044 & LLLLLRRR \\
		5 & 3.82025 & LLLRR & 8 & 5.99732 & LLLLLRLR \\
		5 & 3.86953 & LLRLR & 8 & 6.01003 & LLLLRRRR \\
		6 & 4.41776 & LLLLLR & 8 & 6.03523 & LLLLRLLR \\
		6 & 4.53410 & LLLLRR & 8 & 6.08235 & LLLLRLRR \\
		6 & 4.56631 & LLLRRR & 8 & 6.08382 & LLLLRRLR \\
		6 & 4.59381 & LLLRLR & 8 & 6.10805 & LLLRLRRR  \\
		6 & 4.63714 & LLRLRR & 8 & 6.12145 & LLLRLLRR  \\
		7 & 5.10303 & LLLLLLR & 8 & 6.12233 & LLLRRLLR \\
		7 & 5.23419 & LLLLLRR & 8 & 6.13512 & LLLRRLRR \\
		7 & 5.28634 & LLLLRRR & 8 & 6.15472 & LLLRLRLR \\
		7 & 5.30120 & LLLLRLR & 8 & 6.17587 & LLRLLRLR \\
		7 & 5.33091 & LLLRLLR & 8 & 6.18751 & LLRLRRLR \\
		7 & 5.36988 & LLLRLRR & 8 & 6.19460 & LLRLRLRR \\
		7 & 5.37052 & LLLRRLR &  &  & \\
		\hline
\end{tabular}}
\label{table2}
\end{table}

\begin{figure}[htpb]
\centering
%\setlength{\abovecaptionskip}{10pt}
%\begin{minipage}[htbp]{1\linewidth}
\subfigure[p=2, LR]{\includegraphics[height=0.3\textwidth,width=0.3\textwidth]{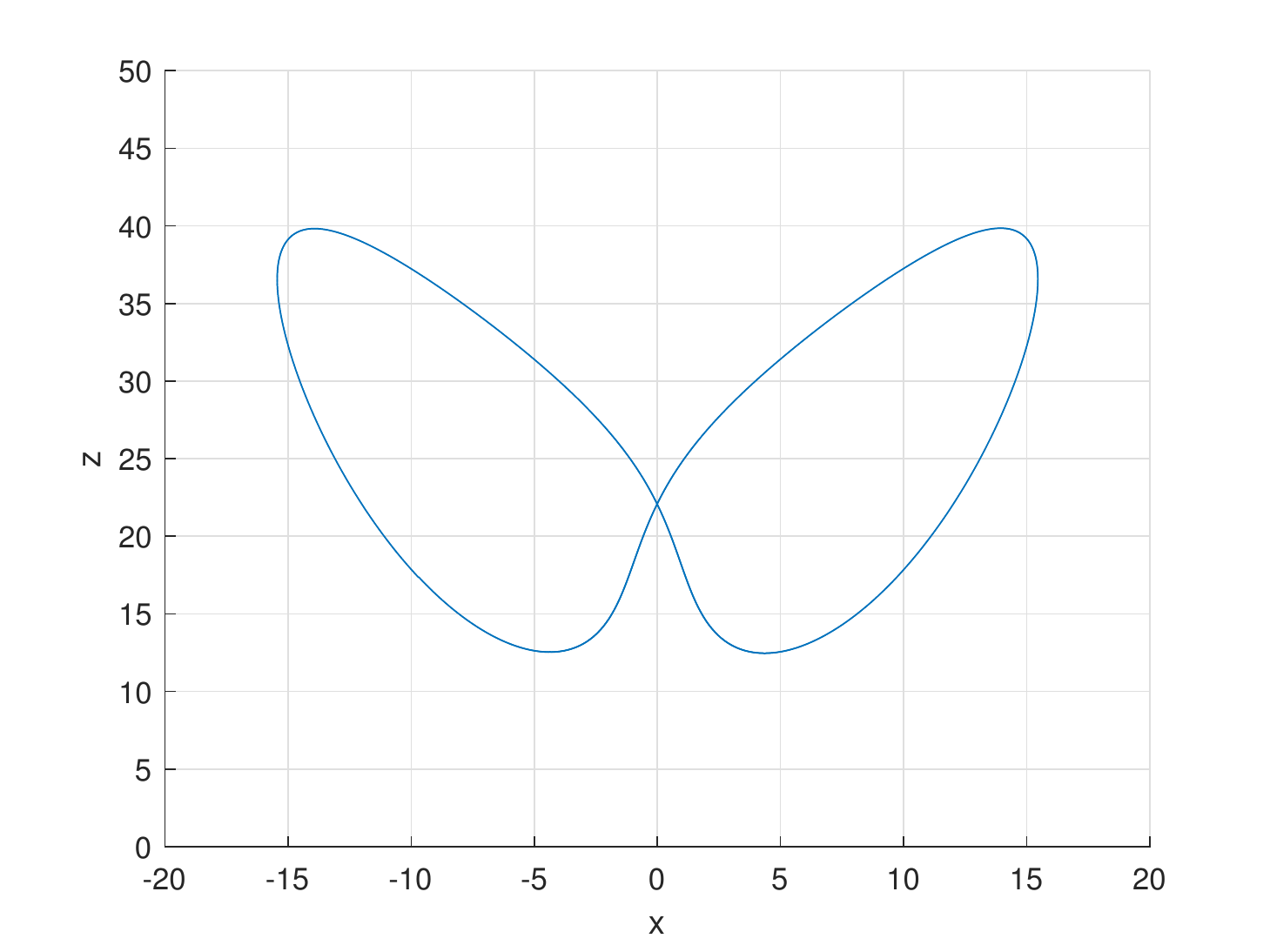}}
\subfigure[p=3, LLR]{\includegraphics[height=0.3\textwidth,width=0.3\textwidth]{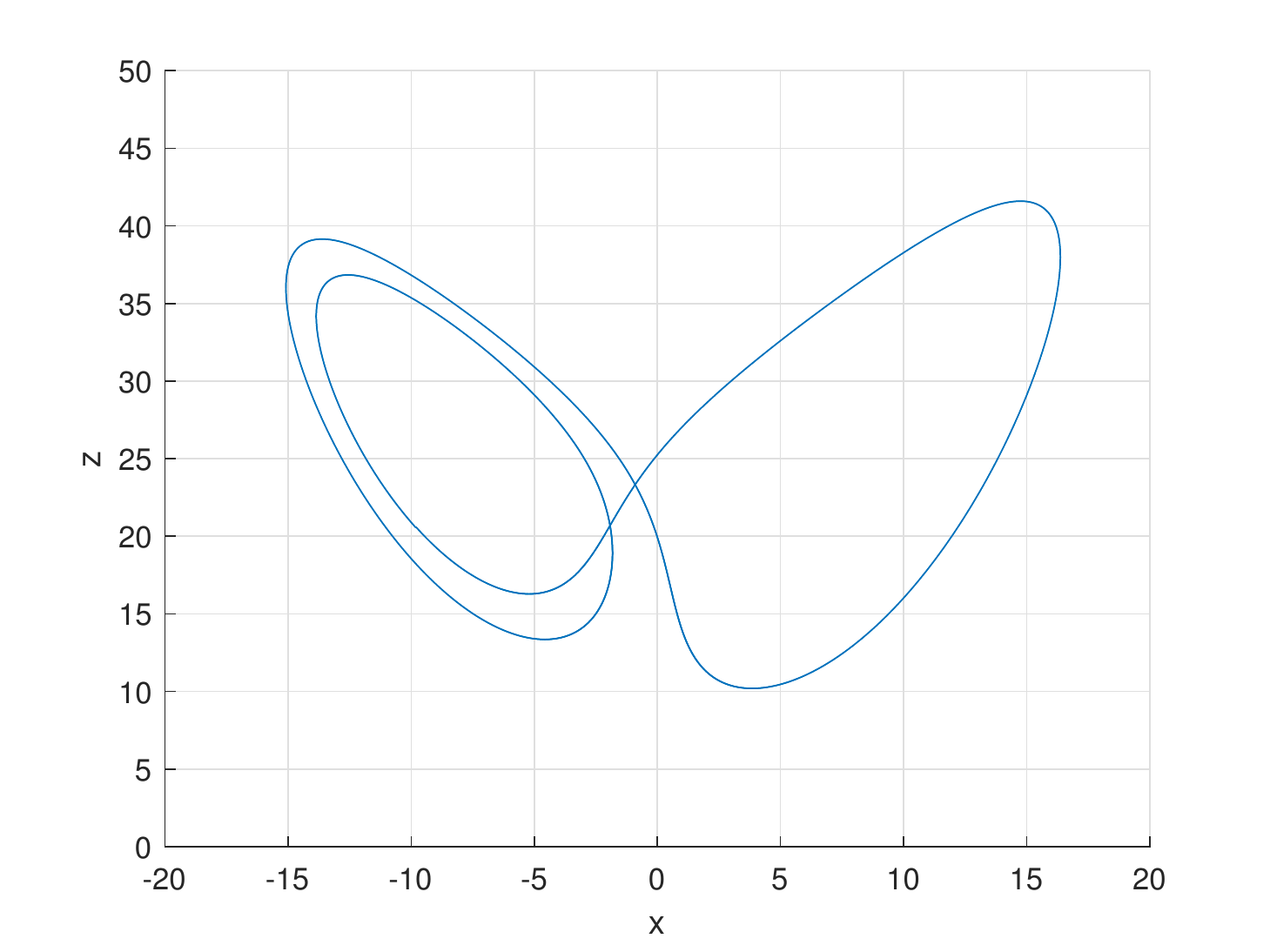}}
\subfigure[p=4, LLRR]{\includegraphics[height=0.3\textwidth,width=0.3\textwidth]{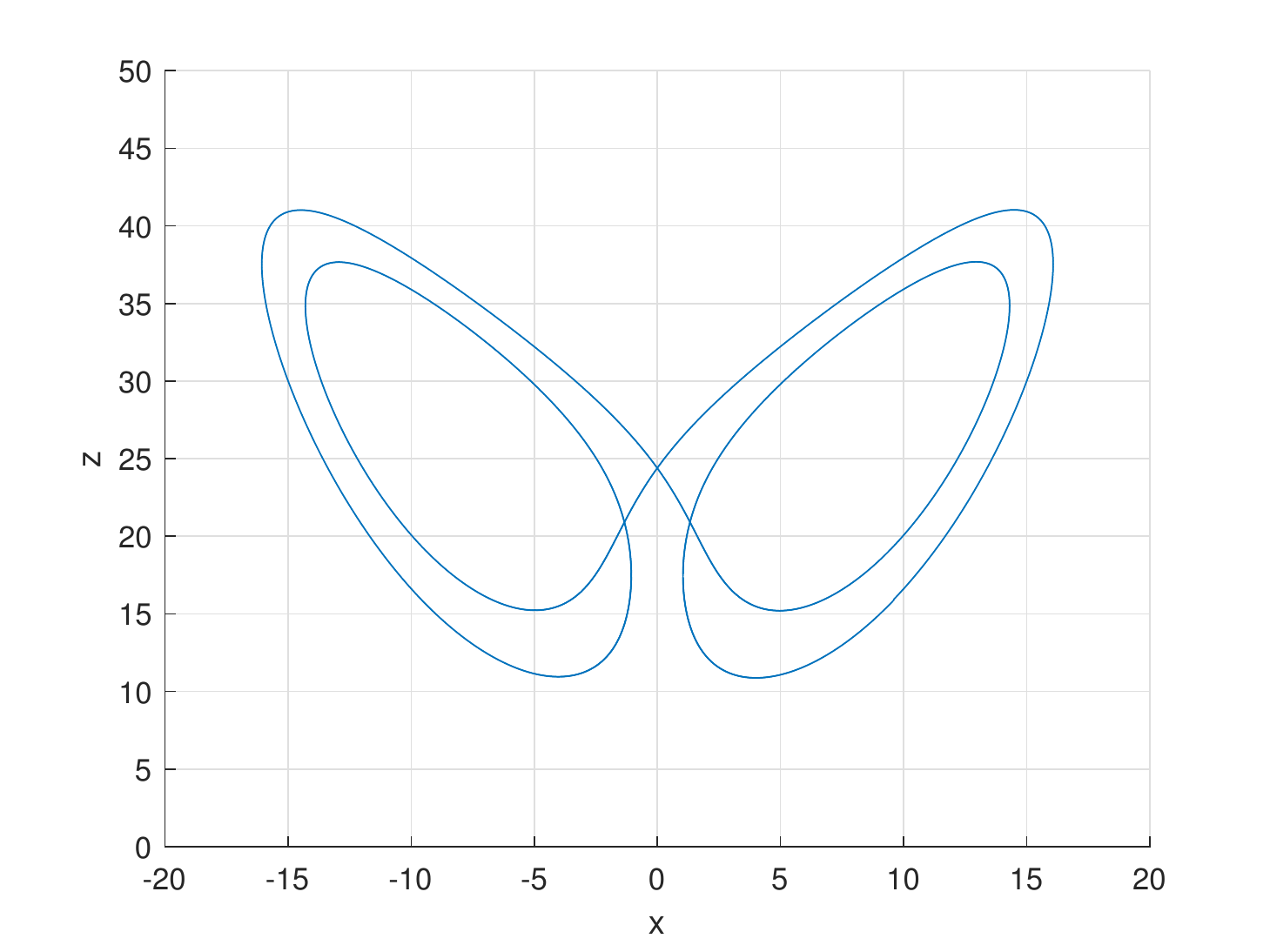}}
\subfigure[p=4, LLLR]{\includegraphics[height=0.3\textwidth,width=0.3\textwidth]{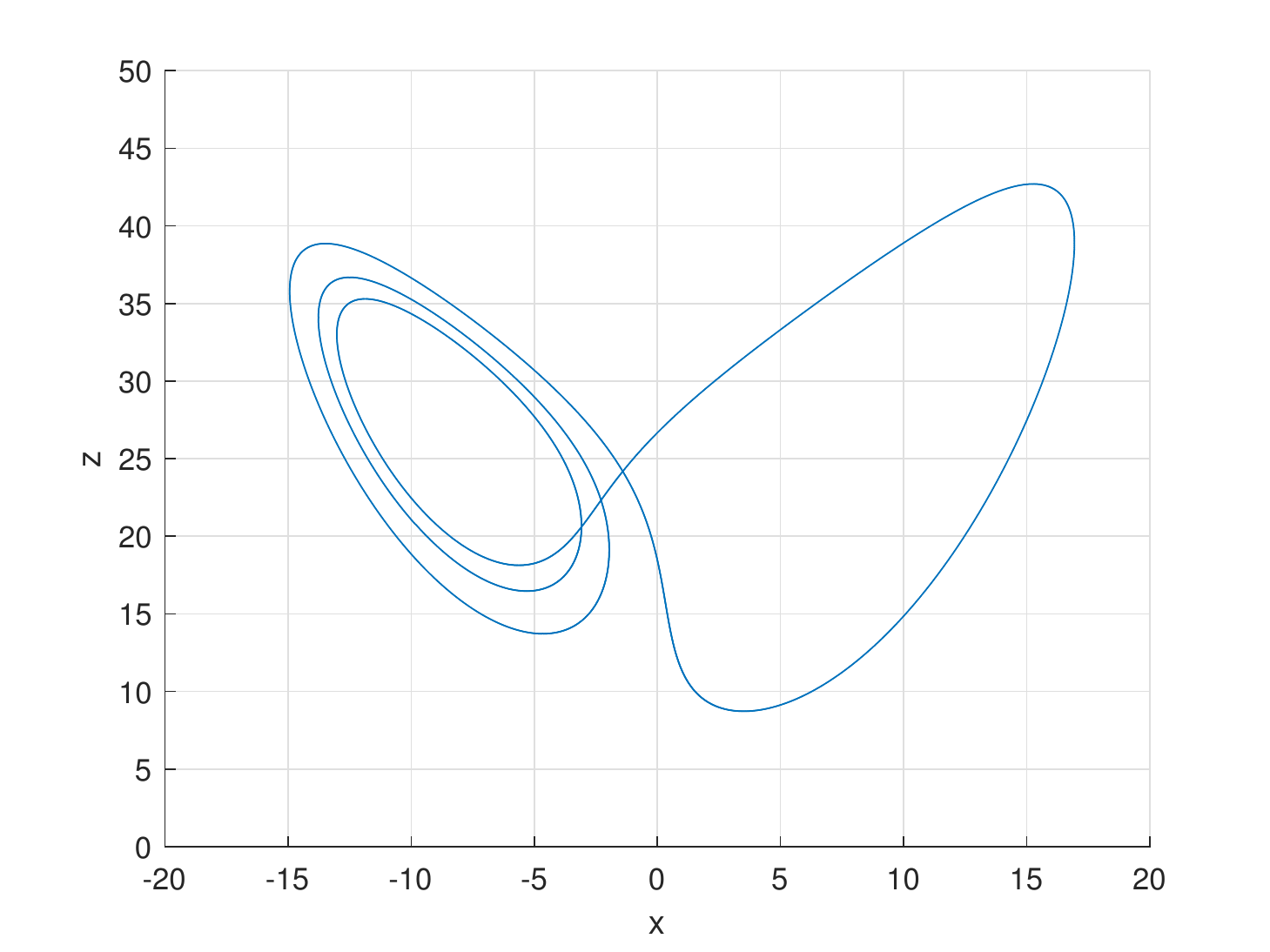}}
\subfigure[p=5, LLLRR]{\includegraphics[height=0.3\textwidth,width=0.3\textwidth]{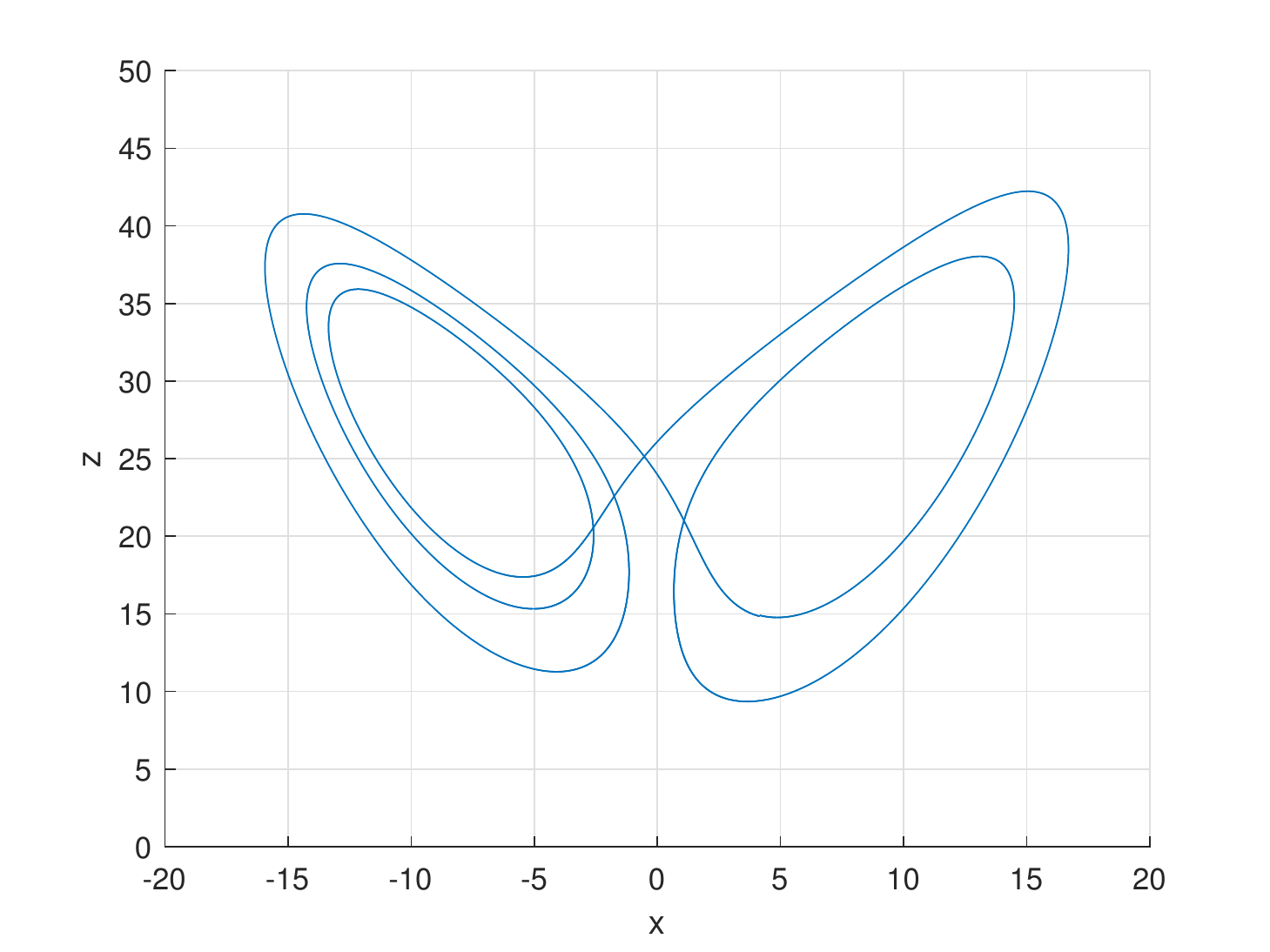}}
\subfigure[p=5, LLLLR]{\includegraphics[height=0.3\textwidth,width=0.3\textwidth]{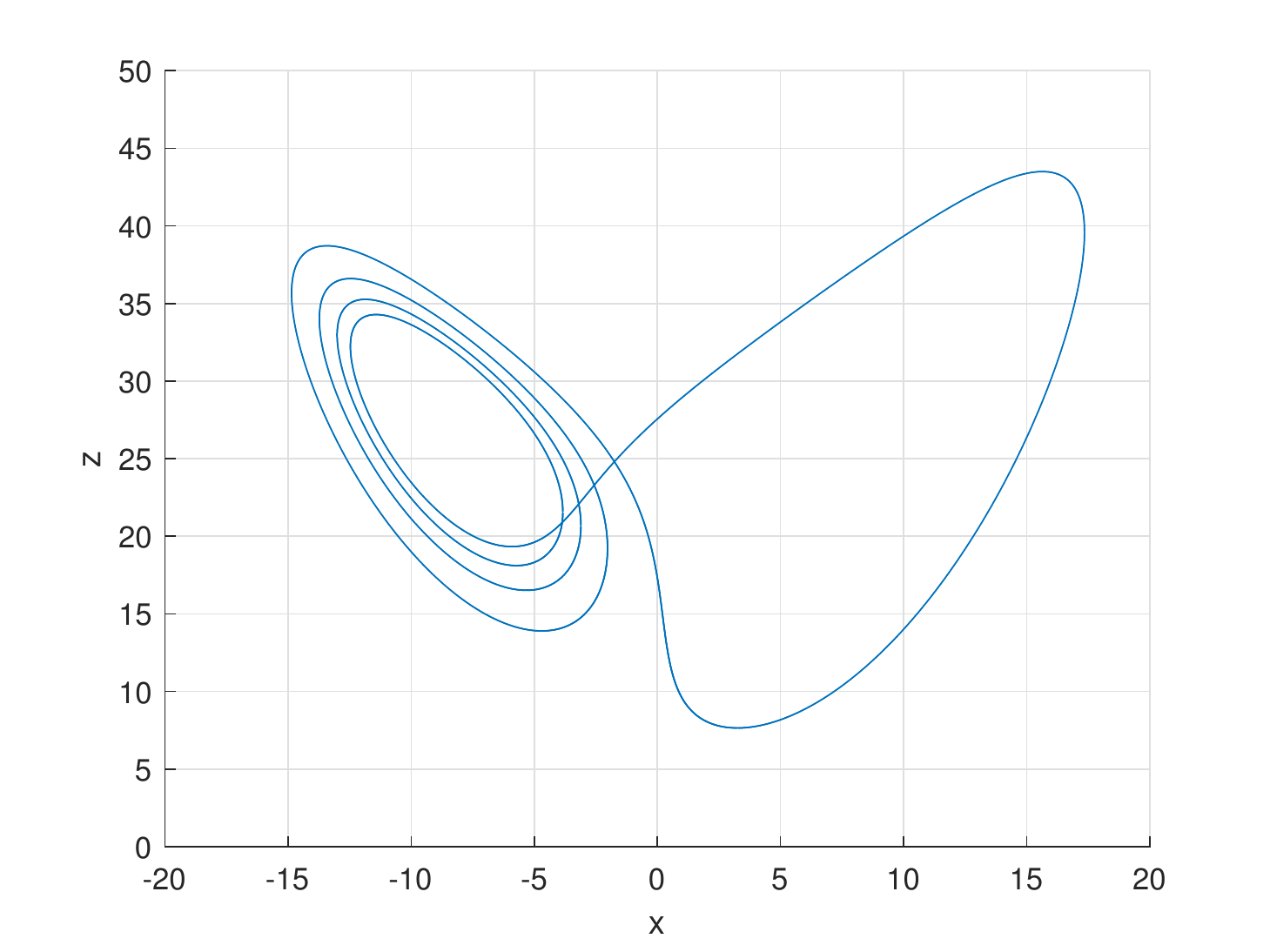}}
\subfigure[p=5, LLRLR]{\includegraphics[height=0.3\textwidth,width=0.3\textwidth]{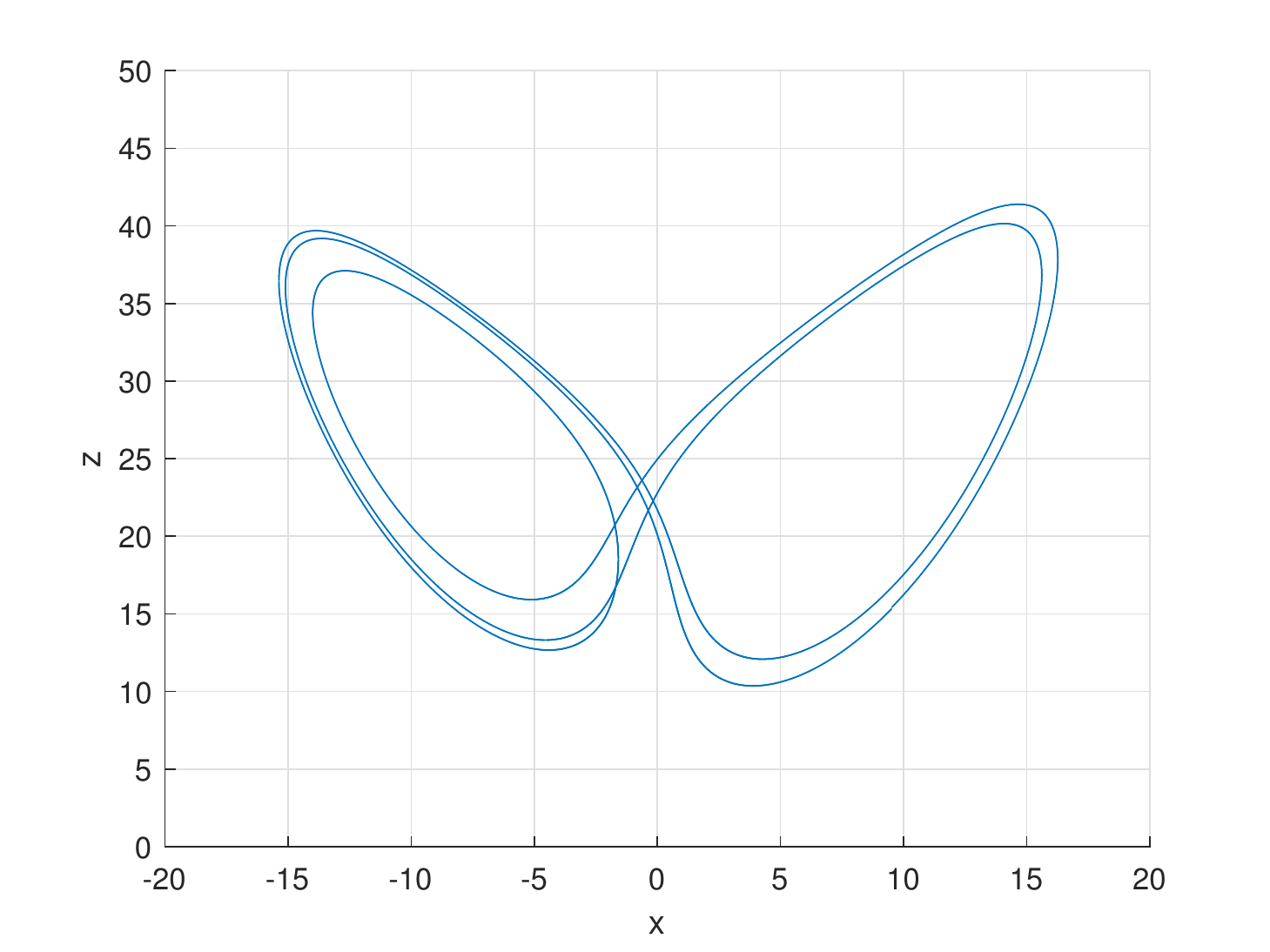}}
\subfigure[p=6, LLLRRR]{\includegraphics[height=0.3\textwidth,width=0.3\textwidth]{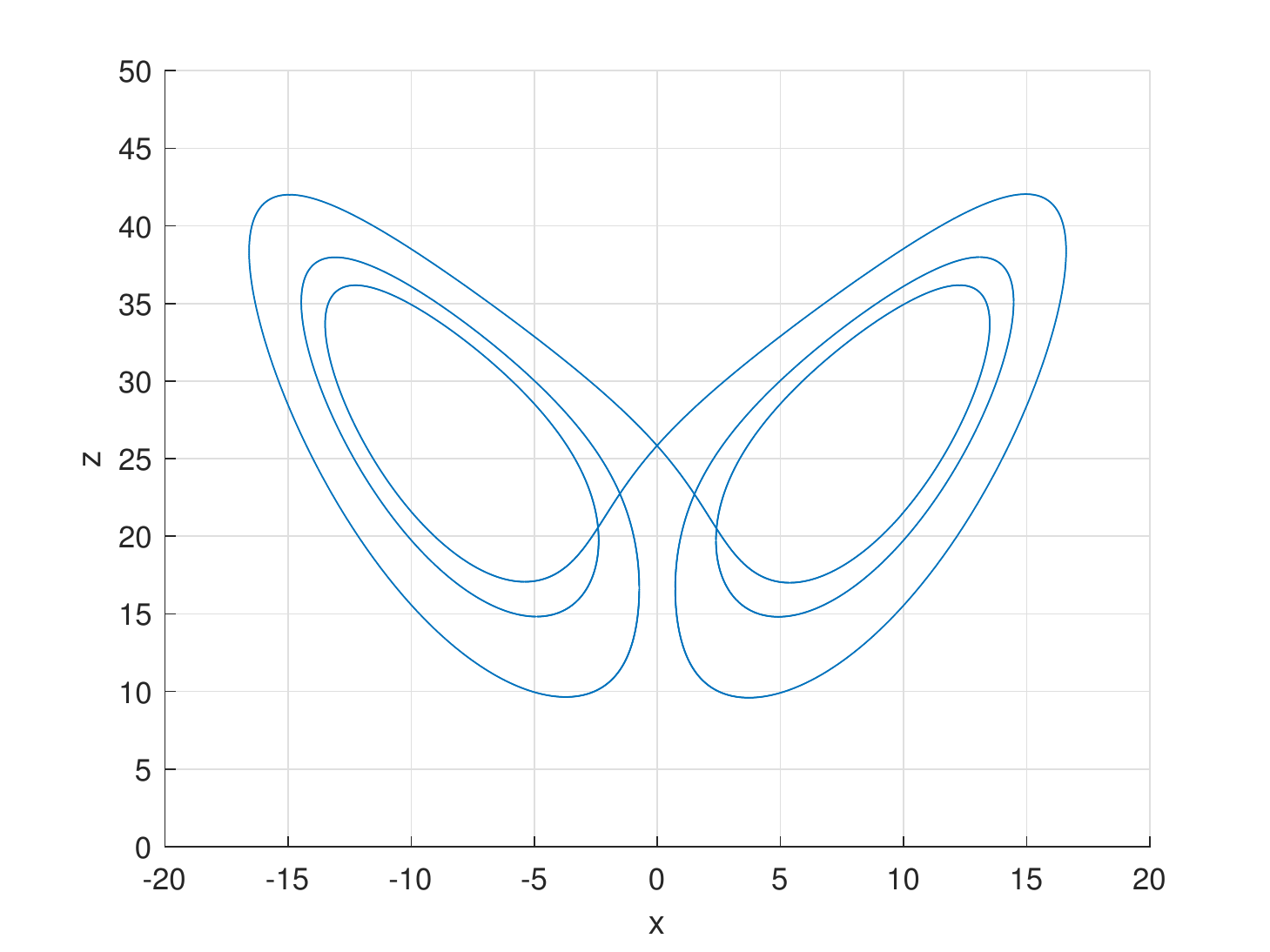}}
\subfigure[p=6, LLLLLR]{\includegraphics[height=0.3\textwidth,width=0.3\textwidth]{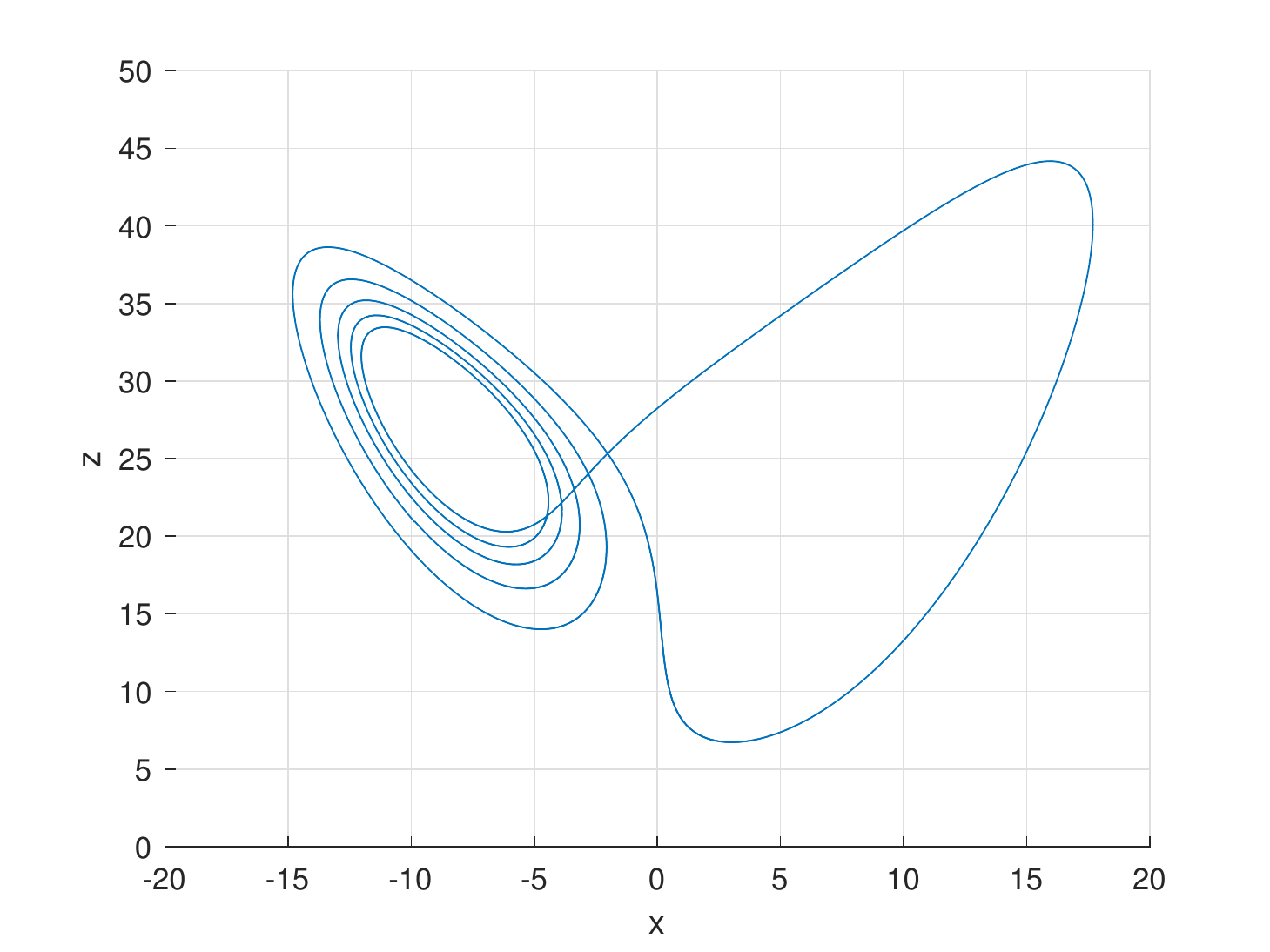}}
\subfigure[p=6, LLLLRR]{\includegraphics[height=0.3\textwidth,width=0.3\textwidth]{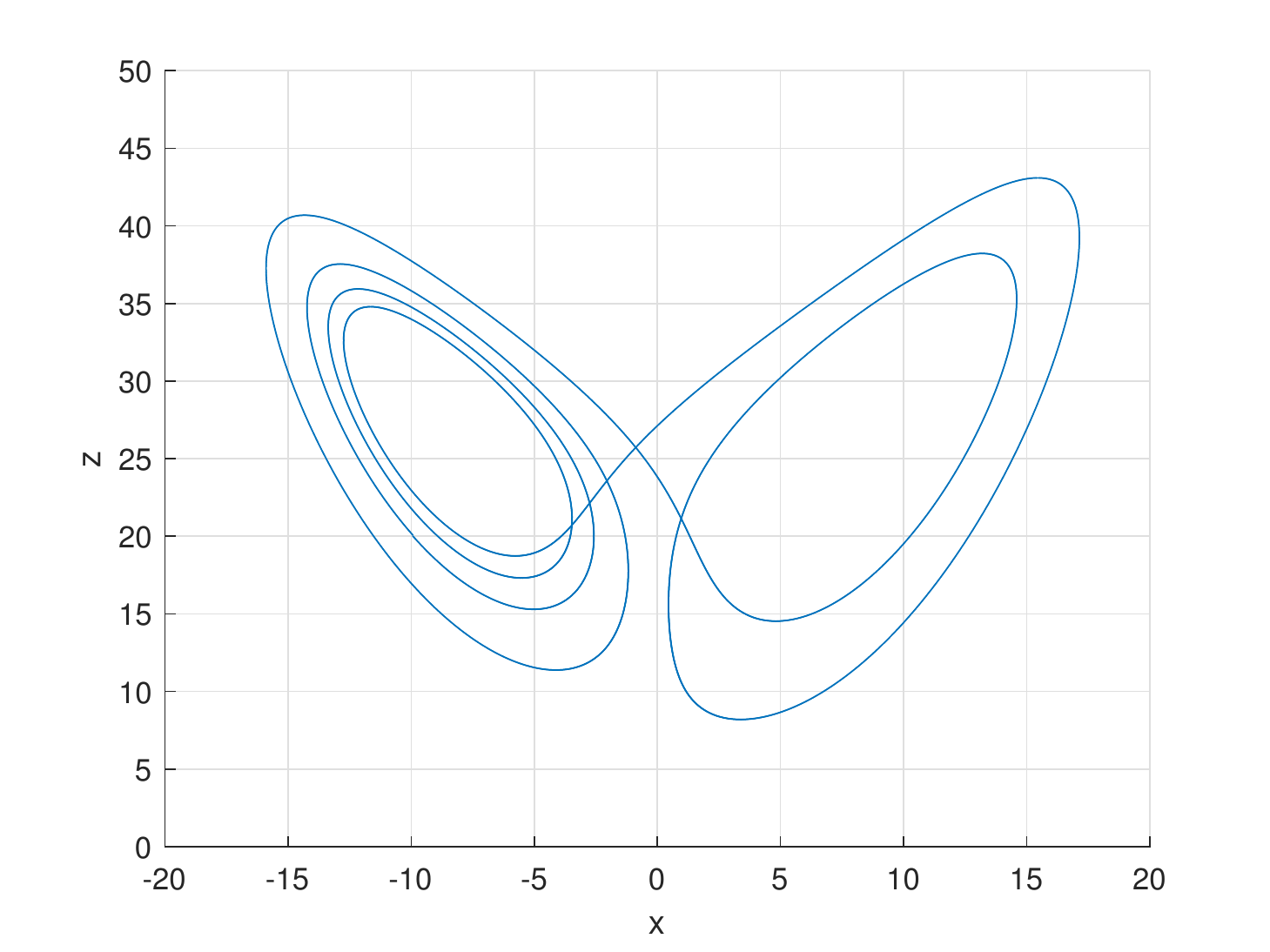}}
\subfigure[p=6, LLLRLR]{\includegraphics[height=0.3\textwidth,width=0.3\textwidth]{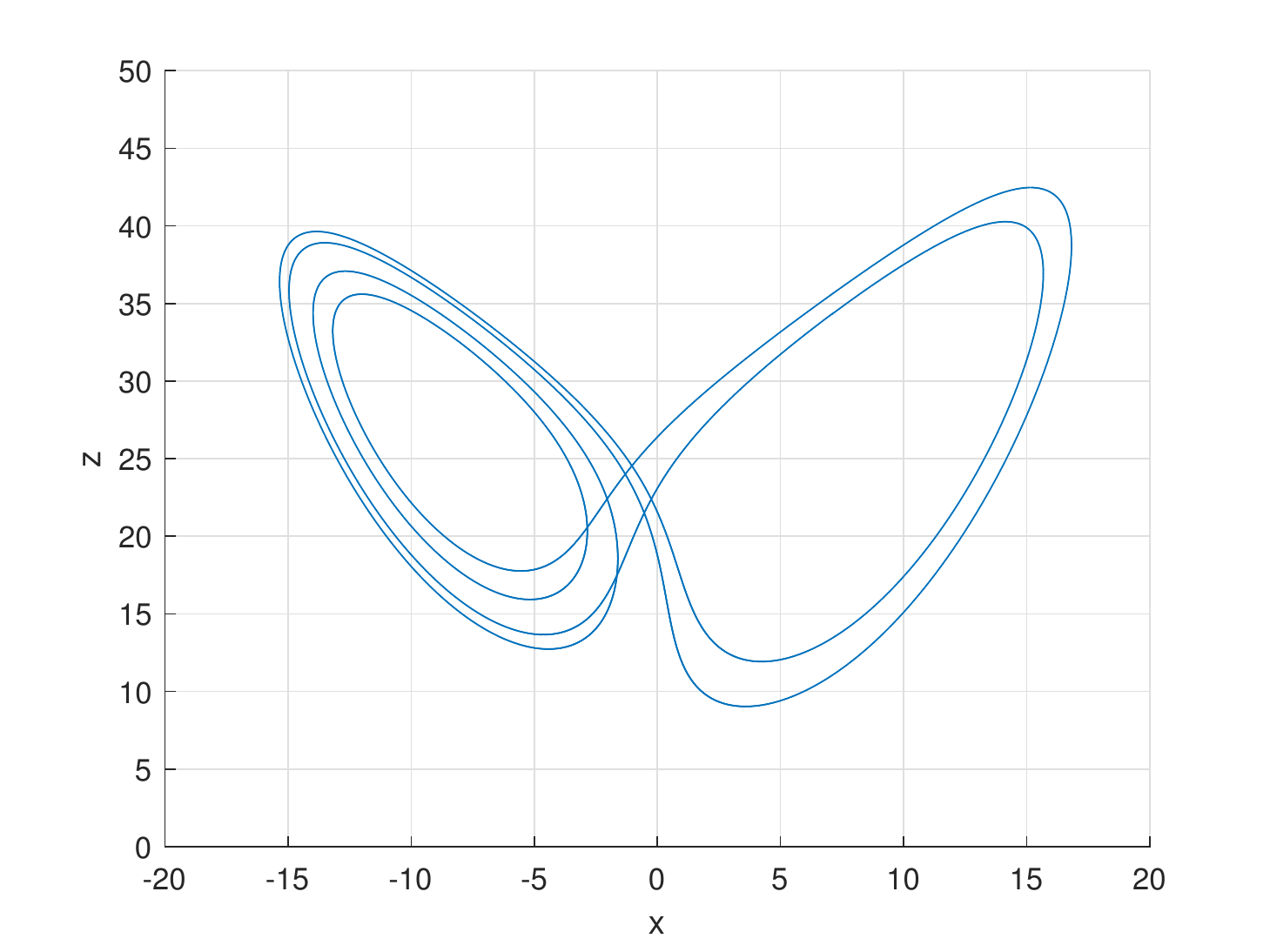}}
\subfigure[p=6, LLRLRR]{\includegraphics[height=0.3\textwidth,width=0.3\textwidth]{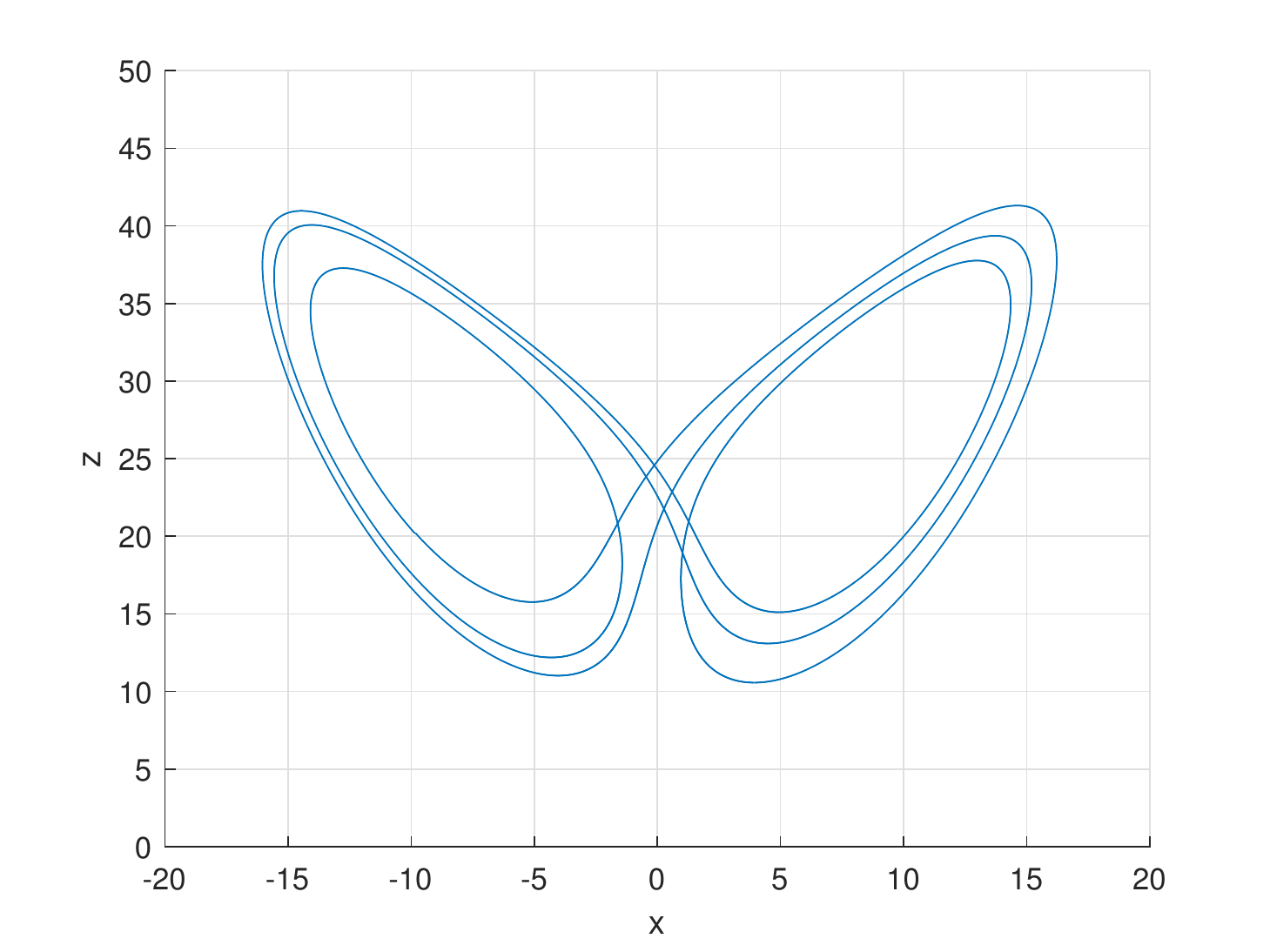}}
\caption{Periodic orbits of the Lorenz system with $p \le 6$, where $p$ is the period, and the order in which L and R appear is the trajectory of the periodic orbit.}
\label{Fig:fige11}
%\end{minipage}
\end{figure}
%\begin{table}[htbp]
%	\centering
%	\setlength{\abovecaptionskip}{0pt}
%	\setlength{\belowcaptionskip}{10pt}
%	\caption{$p$ is the period of the orbit, $T$ is the flow-time, $s$ is the corresponding symbol sequence.}
%	\begin{tabular}[htbp]{llllll}
%		\hline
%		p & ${T}$ & s & p & ${T}$ & s\\
%		\hline
%		2 & 3101 & LR & 5 & 7695 & LLRLR \\
%		3 & 4588 & LLR & 6 & 8797 & LLLLLR \\
%		4 & 6014 & LLLR & 6 & 9034 & LLLLRR \\
%		4 & 6132 & LLRR & 6 & 9086 & LLLRRR \\
%		5 & 7415 & LLLLR & 6 & 9136 & LLLRLR \\
%		5 & 7599 & LLLRR & 6 & 9221 & LLRLRR \\ \hline
%	\end{tabular}
%\end{table}
%\vspace{-0.4cm}

\section{Conclusion}
%This paper presents a new method for finding periodic orbits in the Lorenz system. {\color{red}We} combine the similar signature curve with the sliding window method to find the quasi-periodic orbits and use the interval operator to verify the existence of the real orbits, and successfully find all the short-period orbits with $p \leqslant 6$.

%In this work, {\color{red} we }propose a new method to find system periodic orbits using similar signature curve of Lorenz system. Since the similar signature curves of Lorenz system have some rules, {\color{red} we }used the sliding window method to segment the Lorenz flow and obtain its periodic orbits. {\color{red}We} found all $p \leqslant 6$ periodic orbits in the Lorenz system and verified them with Krawczy operator.
In this paper, a novel method based on the similarity signature curve is proposed for forming the periodic orbits of the Lorenz system, and some experimental results are exhibited to show the efficacy of this method.

Firstly, through the similarity invariants of the Lorenz system and a series of experiments, we can demonstrate the chaotic patterns of the Lorenz system evolution, and ensure the feasibility of the method of applying the similarity signature curve to calculate the periodic orbits.

Secondly, in order to continuously obtain the required quasi-periodic orbits, the conditions for the window moving in the sliding window method are improved. The window is updated by repeatedly selecting the points in the subsequent data which are closest to the initial point on the similarity signature curve.

In this way, all periodic orbits with period $p \leqslant 8$ are found. Also, it is a common fact to predict long-period orbits by using some short-period orbits. However, the longer the period, the larger the error and the more difficult to locate the corresponding position. Searching exact periodic orbits with similarity signature curves offers certain advantages comparing with other methods.

\end{document}